\documentclass[11pt]{amsart}
\usepackage{amsfonts, amsmath, amssymb, epsfig, cite, graphicx, hyperref, esint, fancyhdr, enumerate, latexsym, amsrefs}

\newtheorem{thm}{Theorem}[section]
\newtheorem{prop}[thm]{Proposition}
\newtheorem{lem}[thm]{Lemma}
\newtheorem{cor}[thm]{Corollary}

\theoremstyle{definition}

\newcommand{\ep}{\epsilon}
\newcommand{\mbb}{\mathbb}
\newcommand{\de}{\delta}
\newcommand{\ov}{\overline}

\newcommand{\pa}{\partial}
\newcommand{\mf}{\mathbb}
\newcommand{\om}{\omega}
\newcommand{\Om}{\Omega}
\newcommand{\al}{\alpha}
\newcommand{\be}{\beta}
\newcommand{\ga}{\gamma}

\newcommand{\la}{\lambda}
\newcommand{\ti}{\tilde}
\newcommand{\De}{\Delta}
\newcommand{\Ga}{\Gamma}

\renewcommand{\Re}{\operatorname{Re}}
\renewcommand{\Im}{\operatorname{Im}}

\numberwithin{equation}{section}
\hypersetup{
    colorlinks,
    citecolor=blue,
    filecolor=black, 
    linkcolor=red,
    urlcolor=black
}

\title{Some remarks on the Carath\'eodory and Szeg\"o metrics on planar domains}
\keywords{Carath\'eodory metric, Szeg\"o metric, Gaussian curvature}
\subjclass{Primary: 32F45; Secondary: 32A25, 30F45}
\author{Anjali Bhatnagar and Diganta Borah}

\address{Anjali Bhatnagar: Indian Institute of Science Education and Research Pune, Pune~411008, India}
\email{anjali.bhatnagar@students.iiserpune.ac.in}

\address{Diganta Borah: Indian Institute of Science Education and Research, Pune~411008, India}
\email{dborah@iiserpune.ac.in}

\begin{document}

\begin{abstract}
We study several intrinsic properties of the Carath\'eodory and Szeg\"o metrics on finitely connected planar domains. Among them are the existence of closed geodesics and geodesic spirals, boundary behaviour of Gaussian curvatures, and $L^2$-cohomology. A formula for the Szeg\"o metric in terms of the Weierstrass $\wp$-function is obtained. Variations of these metrics and their Gaussian curvatures on planar annuli are also studied. Consequently, we obtain optimal universal upper bounds for their Gaussian curvatures and show that no universal lower bounds exist for their Gaussian curvatures. Moreover, it follows that there are domains where the Gaussian curvature of the Szeg\"o metric assumes both negative and positive values. Lastly, it is also observed that there is no universal upper bound for the ratio of the Szeg\"o and Carath\'eodory metrics.
\end{abstract}
\maketitle

\section{Introduction}

The purpose of this article is to study the Carath\'eodory and Szeg\"o metrics on non-degenerate finitely connected planar domains. The Carath\'eodory metric on a domain $\Omega \subset \mf{C}$ is the function $c_{\Om}: \Omega \to [0,\infty)$ defined by
\[
c_{\Om}(z)= \sup \Big\{ \big\vert f'(z)\big\vert : f: \Om \to  \mf{D} \text{ holomorphic and } f(z)=0\Big\}.
\]
Under a conformal equivalence $\phi: \Om \to \Om'$, $c_{\Om}(z)$ transforms by the rule
\[
c_{\Om}(z)=c_{\Om'}\big(\phi(z)\big)\big\vert \phi'(z)\big\vert.
\]
Suita \cite{Suita-73-1} showed that if $\Om$ admits a nonconstant bounded holomorphic function then $c_{\Om}$ is a positive real analytic function and thus,
\[
ds_{c_{\Om}}=c_{\Om}(z) \vert dz\vert,
\]
is a conformal metric on $\Om$. The Carath\'eodory metric is analogously defined in higher dimensions. It is invariant under biholomorphisms but not smooth, in fact not even continuous, in general.

The Szeg\"o metric on a $C^{\infty}$-smoothly bounded domain $\Om \subset \mf{C}$ is defined as follows. Denote by $b\Om$ the boundary of $\Om$ and let $L^2(b \Om)$ be the Hilbert space of complex-valued functions on $b \Om$ that are square integrable with respect to the arc length measure $ds$ on $b\Om$. The Hardy space $H^2(b\Om)$ is the closure in $L^2(b\Om)$ of the set of boundary traces of holomorphic functions on $\Om$ which are smooth on $\ov \Om$. Every element in $H^2(b\Om)$ has a (necessarily unique) holomorphic extension to $\Om$ given by its Poisson integral and for brevity we will identify the elements of $H^2(b\Om)$ with their holomorphic extensions. The space $H^2(b\Om)$ is a reproducing kernel Hilbert space and the associated reproducing kernel is called the Szeg\"o kernel for $\Om$. It is denoted by $S_{\Om}(z,\zeta)$ and is uniquely determined by the following properties: (i) for each $\zeta \in \Om$, $S_{\Om}(\cdot,\zeta) \in H^2(b\Om)$, (ii) for each $z,\zeta \in \Om$, $S_{\Om}(z,\zeta)=\ov{S_{\Om}(\zeta,z)}$, and (iii) for each $f \in H^2(b\Om)$ and $z \in \Om$,
\[
f(z)=\int_{b \Om} S_{\Om}(z, \zeta) f(\zeta) \, ds.
\]
In terms of an orthonormal basis $(\phi_j)_{j\geq 1}$ of $H^2(b \Om)$, one has
\[
S_{\Om}(z,\zeta)=\sum_{j=1}^{\infty} \phi_j(z) \ov {\phi_j(\zeta)}, \quad z, \zeta \in \Om,
\]
where the convergence is uniform on compact subsets of $\Om \times \Om$. We recall from \cite{bl16} that under a conformal equivalence $\phi: \Om \to \Om'$ between two bounded domains with $C^{\infty}$-smooth boundaries, the Szeg\"o kernel transforms according to the rule
\begin{equation}\label{trans}
S_{\Om}(z,\zeta)=\sqrt{\phi'(z)}S_{\Om'}\big(\phi(z),\phi(\zeta)\big)\ov{\sqrt{\phi'(\zeta)}}.
\end{equation}
The function $\phi'$ has a single-valued square root on $\Om$ and thus $\sqrt{\phi'(z)}$ is well-defined. The restriction of the Szeg\"o kernel to the diagonal, $S_{\Om}(z)=S_{\Om}(z,z)$, is a positive real analytic function on $\Om$. By a classical result of Garabedian, the Carath\'eodory metric and the Szeg\"o kernel on a $C^{\infty}$-smoothly bounded domain $\Om\subset \mf{C}$ are related by the identity $c_{\Om}(z)=2\pi S_{\Om}(z)$ (see for example \cite[p. 118]{bg70}), using which the real analyticity of the Carath\'eodory metric was established in \cite{Suita-73-1}.  The function $\log S_{\Om}(z)$ is strictly subharmonic. It follows from \eqref{trans} that
\[
ds_{s_{\Om}}=s_{\Om}(z) \vert dz\vert,
\]
where
\[
s_{\Omega}(z)= \sqrt{\frac{\pa ^2  \log S_{\Om}(z)}{\pa z \pa \ov z}},
\]
is a conformal metric on $\Om$ which is called the Szeg\"o metric on $\Om$.

The Szeg\"o kernel and the Szeg\"o metric are defined on any finitely connected planar domain $\Om$ (not necessarily $C^{\infty}$-smoothly bounded) such that no boundary component of $\Om$ is a singleton (see for example \cite{Bell-2002}). Recall that such a domain is called a non-degenerate finitely connected domain. It is well-known that there is a conformal equivalence $\phi: \Om \to \Om'$ where $\Om' \subset \mf{C}$ is a bounded domain with real analytic boundary. Also, the function $\phi'$ has a single-valued square root on $\Om$. It is thus customary to define the Szeg\"o kernel $S_{\Om}(z,w)$ with the aid of the transformation formula \eqref{trans}. Then the Szeg\"o metric $s_{\Om}(z)\vert dz\vert$ is defined analogously.

The Szeg\"o kernel and the Szeg\"o metric are defined on $C^{\infty}$-smoothly bounded domains in higher dimensions using the Euclidean surface area measure. However, in dimension $n \geq 2$, the Euclidean surface area measure does not transform appropriately under a biholomorphic mapping and as a result, the Szeg\"o kernel does not obey a transformation rule. To resolve this issue, Fefferman constructed the so-called Fefferman surface area measure (page 259 of \cite{f79}) using which Barret and Lee \cite{bl14} introduced an invariant Ka\"hler metric---which is called the Fefferman-Szeg\"o metric---and studied various properties of this metric on strongly pseudoconvex domains. This metric was further investigated by Krantz in \cite{k19} and \cite{k21}. We note that in dimension $n=1$, the Fefferman surface area measure reduces to the arc length measure and hence, the Fefferman-Szeg\"o metric coincides with the Szeg\"o metric.

Biholomorphic invariants, and in particular, their boundary behaviour play an important role in understanding the geometry of domains. The classical invariant metrics such as the Kobayashi, Carath\'eodory, and Bergman metrics have been extensively studied, both in one and higher dimensions, and they have found far-reaching applications in several areas of mathematics. In this article, we study various intrinsic properties of the Carath\'eodory and Szeg\"o metrics such as geodesics, curvature, $L^2$-cohomology etc. and show that these metrics resemble the Bergman metric on planar domains. We use the scaling method to compute the boundary asymptotics of the Carath\'eodory and Szeg\"o metrics and utilise them to investigate the above-mentioned properties. We will also study variations of these metrics on planar annuli and answer some questions related to their comparability and universal bounds for their Gaussian curvatures.

\subsection{Geodesics}
The geodesics in the Bergman metric escaping towards the boundary plays a crucial role in Fefferman's proof of the smooth extension up to the boundary of biholomorphic mappings between $C^{\infty}$-smoothly bounded strongly pseudoconvex domains \cite{f74}. It naturally leads to the following question: Does there exist a geodesic for the Bergman metric that stays within a compact subset of a $C^{\infty}$-smoothly bounded strongly pseudoconvex domain? Such a geodesic (if it exists) can be closed or non-closed. The latter one is known as a \emph{geodesic spiral}. Herbort investigated this question in \cite{hg83} and provided an affirmative answer for domains that are not topologically trivial. On the other hand, no non-trivial closed geodesics or geodesic spirals exist in a simply connected planar domain that is not all of $\mf{C}$ as such a domain is holomorphically equivalent to the unit disc and the Bergman metric on the unit disc coincides with the Poincar\'e metric. Results of this kind were obtained for the capacity metric in \cites{bhv18} and for the Kobayashi-fuks metric in \cites{bk22, k23}. We prove that the analogous results hold for the Carath\'eodory and Szeg\"o metrics:

\begin{thm}\label{geo}
Let $\Omega \subset \mf{C}$ be a non-degenerate $n$-connected domain, $n\geq 2$. Equip $\Om$ with the conformal metric $ds_{m_{\Om}}=m_{\Omega}(z)\vert dz\vert$, where $m_{\Omega}=c_{\Omega}$ or $s_{\Omega}$. Then
\begin{enumerate}
\item[(i)] Every non-trivial homotopy class of loops in $\Omega$ contains a closed geodesic.

\item[(ii)] For every $z_0 \in \Omega$ that does not lie on a closed geodesic, there exists a geodesic spiral passing through $z_0$.
\end{enumerate}
\end{thm}

\subsection{Curvature}
The Gaussian curvature of a $C^2$-smooth conformal metric $m_{\Omega}(z) \vert dz\vert$ on a domain $\Om \subset \mf{C}$ is defined by
\[
\kappa_{m_{\Om}}(z)=-\frac{\De \log m_{\Om}(z)}{m_{\Om}^2(z)}.
\]
As there are different normalisations of Gaussian curvatures across the literature, we note here for clarity that on the unit disc $\mf{D}$,
\[
c_{\mf{D}}(z)=s_{\mf{D}}(z)=\frac{1}{1-\vert z \vert^2},
\]
and $\kappa_{c_\mf{D}}(z)=\kappa_{s_\mf{D}}(z)=-4$ for all $z \in \mf{D}$. Suita \cite{Suita-73-1} showed that $\kappa_{c_\Om}$ is at most $-4$ on any domain $\Om \subset \mf{C}$ that admits a nonconstant bounded holomorphic function. On the other hand, Burbea \cite{Burbea-77} showed that if $\Om \subset \mf{C}$ is a $C^2$-smoothly bounded domain, then $\kappa_{c_\Om}(z)$ approaches $-4$ if $z$ approaches $b\Om$ nontangentially. In \cite{sv20}, it was obtained that $\kappa_{c_\Om}(z)$ approaches $-4$ without any restricted approach to the boundary. The limiting behaviour of the higher-order curvatures of the Caratheodory metric was also studied in the above article. For the Szeg\"o metric, we will see (in Section~5) that $\kappa_{s_{\Om}}$ is at most $4$ on any non-degenerate finitely connected domain $\Om \subset \mf{C}$. As for the limiting behaviour of $\kappa_{s_{\Om}}$, let us first define, following \cite{Bur-curv-77}, the $N$-th order curvature, $N \geq 1$, of the metric $ds_{m_{\Om}}=m_{\Om}(z)\vert dz \vert$ by
\[
\kappa^{(N)}_{m_{\Om}}(z)=-4\frac{\det \begin{pmatrix}\pa^{j\ov k} m_{\Om}(z)\end{pmatrix}_{j,k=0}^N}{m_{\Om}^{(N+1)^2}(z)},
\]
where $\pa^j=\pa^j/\pa z^j$ , $\pa^{\ov k}=\pa^k/\pa \ov z^k$, and $\pa^{j\ov k}=\pa^j\pa^{\ov k}$. The function $\kappa^{(N)}_{m_{\Om}}(z)$ is invariant under a conformal equivalence and $\kappa^{(1)}_{m_{\Om}}(z)=\kappa_{m_{\Om}}(z)$. We have the following boundary behaviour of the higher order Gaussian curvatures:
\begin{thm}\label{bdry curv}
Let $\Omega \subset \mf{C}$ be a non-degenerate finitely connected domain and let $m_{\Omega}=c_{\Omega}$ or $s_{\Omega}$. Then for every $p\in b\Omega$,
\[
\kappa^{(N)}_{m_{\Omega}}(z)\to -4\left(\prod_{m=1}^{N}m! \right)^2,
\] 
as $z \to p$.
\end{thm}
In particular, Theorem \ref{bdry curv} implies that $\kappa_{s_{\Om}}(z)$ approaches $-4$ as $z$ approaches $b\Om$. This, combined with \cite[Theorem 1.17]{gk89}, immediately gives the following:

\begin{cor}\label{iso-hol}
Let $\Om_{1}$ and $\Om_2$ be two non-degenerate finitely connected planar domains equipped with the metrics $ds_{c_{\Om_1}}$ and $ds_{c_{\Om_2}}$ (resp. $ds_{s_{\Om_1}}$ and $ds_{s_{\Om_2}}$). Then, each isometry $f: (\Om_1, ds_{c_{\Om_1}})\to (\Om_2, ds_{c_{\Om_2}})$ (resp. $f: (\Om_1, ds_{s_{\Om_1}})\to (\Om_2, ds_{s_{\Om_2}})$) is either holomorphic or conjugate holomorphic.
\end{cor}

\subsection{$L^2$-cohomology}
Given a complete K\"ahler manifold the question of whether there are nontrivial square integrable harmonic forms is of interest because every $L^2$-cohomology class has a harmonic representative---which is an analogue of the Hodge theorem for non-compact manifolds. This question for the Bergman metric on strongly pseudoconvex domains in $\mbb C^n$ was studied by Donnelly--Fefferman \cite{df83} and Donnelly \cite{d94} and in a more general setup by McNeal \cite{m02} and Ohsawa \cite{o89} among others. Let us fix some notation to state our next result. Denote by ${\Om}^{k}_2$ the space of $k$-forms on $\Om$ which are square integrable with respect to $ds_{m_{\Om}}=m_{\Om}(z)\vert dz\vert$ where $m_{\Om}=c_{\Om}$ or $s_{\Om}$. Then the $L^2$-cohomology of the complex
\[
\Om_2^0 \xrightarrow{d_0} \Om_2^1 \xrightarrow{d_1} \Om_2^2\xrightarrow{d_2} 0,
\]
is defined by
\[
H^k_2(\Om)=\ker d_k / \ov{\text{Im}(d_{k-1})},
\]
where the closure is taken in the $L^2$-norm. Since $ds_{m_{\Om}}$ is complete, $H^k_2(\Om) \cong \mathcal{H}_2^k(\Om)$, the space of square integrable harmonic forms. We also note the decomposition
\[
\mathcal{H}^k_2(\Om)= \oplus_{p+q=k}\mathcal{H}^{p,q}_2(\Om).
\]

\begin{thm}\label{l2-co}
Let $\Om \subset \mf{C}$ be a non-degenerate finitely connected domain equipped with $m_{\Om}(z)\vert dz\vert$, where $m_{\Om}=c_{\Om}$ or $s_{\Om}$. Then
\[
\dim \mathcal{H}^{p,q}_2(\Om)= \begin{cases}
0 & \text{if $p+q \neq 1$},\\
\infty & \text{if $p+q =1$}.
\end{cases}
\]
\end{thm}

\subsection{The Szeg\"o metric on a doubly connected domain}
In \cite{Zara34}, Zarankiewicz derived a formula for the Bergman kernel on an annulus in terms of the Weierstrass $\wp$-function. Using this, a similar formula for the Szeg\"o kernel on an annulus can be obtained (see for example \cite{Burbea-77}). Using them, formulas for the Bergman and capacity metrics in terms of the $\wp$-function can be derived and were useful in studying the qualitative behaviour of geodesics and curvatures of these metrics on an annulus (see \cite{h83}, \cite{hg83}, and \cite{a05}). We show that the Szeg\"o metric on an annulus has the following expression:

\begin{prop}\label{sgo ann}
Let $r \in (0,1)$ and $A_r=\{z \in \mf{C}: r<\vert z \vert<1\}$. Then the Szeg\"o metric on $A_{r}$ is given by
\[
ds_{s_{A_{r}}}^{2}=\frac{\wp\big(2\log \vert z \vert\big)-\wp\big(2\log \vert z \vert+\om_1+\om_3\big) }{\vert z \vert^2}|dz|^{2},
\]
where $\wp$ is the Weierstrass elliptic $\wp$ function with periods $2\om_1=-2\log r$ and $2\om_3=2i\pi$.
\end{prop}
We hope this formula will be useful in studying the qualitative behaviour of geodesics and curvature of the Szeg\"o metric on an annulus as in the case of the Bergman and capacity metrics.

\subsection{Variations of the Carath\'eodory and Szeg\"o metrics on planar annuli}
Recall that the Gaussian curvatures of the Carath\'eodory and the Szeg\"o metrics have the universal upper bounds $-4$ and $4$ respectively. Theorem~\ref{bdry curv} shows that for the Carath\'eodory metric, the upper bound $-4$ is optimal. It is only natural to ask if for the Szeg\"o metric the upper bound $4$ is optimal. Similarly, it is natural to ask if the Gaussian curvatures of these metrics have universal lower bounds. These questions for the Bergman metric have been studied by several authors---see \cite{hg07}, \cite{l71}, \cite{cl09}, \cite{d10}, and \cite{z10}. Among them, in \cite{d10}, Dinew studied the variation of the Gaussian curvatures of the Bergman metric on planar annuli to answer these questions which was later simplified by Zwonek in \cite{z10}. Zwonek's idea was to analyse the maximal domain functions that appear in the Bergman-Fuks formula for the curvature of the Bergman metric. Using similar ideas we study the variations of the Carath\'eodory and Szeg\"o metrics and their Gaussian curvatures on planar annuli.

\begin{thm}\label{variation}
For $r \in (0,1)$, let $A_r=\{z \in \mf{C}: r<\vert z \vert <1\}$. Then
\begin{itemize}
\item [(a)] 
\[
\lim_{r \to 0+} c_{A_r}(r^{\la})=2\pi S_{A_r}(r^{\la})=
\begin{cases}
1 & 0<\la < 1/2,\\
2 & \la=\frac{1}{2},\\
+\infty & 1/2 < \la <1.
\end{cases}
\]

\item [(b)]
\[
\lim_{r \to 0+} s_{A_r}(r^{\la})=
\begin{cases}
1 & 0<\la<\frac{1}{4}, \\
\sqrt{2} & \la=\frac{1}{4},\\
+\infty & \frac{1}{4} < \la <1.
\end{cases}
\]

\item [(c)]
\[
\lim_{r \to 0+}\kappa_{c_{A_r}}(r^{\la})= \begin{cases}
-4 & 0<\la <\frac{1}{4},\\
-8 & \la=\frac{1}{4},\\
-\infty & \frac{1}{4}< \la < \frac{3}{4},\\
-8 & \la=\frac{3}{4},\\
-4 & \frac{3}{4}<\la <1.
\end{cases}
\]
\item [(d)]
\[
\lim_{r \to 0+}\kappa_{s_{A_r}}(r^{\la})=
\begin{cases}
-4 & 0<\la < \frac{1}{6},\\
-12 & \la=\frac{1}{6},\\
-\infty & \frac{1}{6}< \la <\frac{1}{3},\\
-4 & \la=\frac{1}{3},\\
4 & \frac{1}{3}<\la < \frac{2}{3},\\
-4 & \la=\frac{2}{3},\\
-\infty & \frac{2}{3}< \la <\frac{5}{6},\\
-12 & \la=\frac{5}{6},\\
-4 & \frac{5}{6}<\la < 1.\\
\end{cases}
\]
\end{itemize}
\end{thm}
Some remarks are in order. First, as $r$ decreases to $0$, the annuli $A_r$ exhaust the punctured unit disc $\mf{D}^{*}=\mf{D} \setminus \{0\}$. Therefore, $c_{A_r}$ converges uniformly on compact subsets of $\mf{D}^{*}$ to $c_{\mf{D}^{*}}=c_\mf{D}\vert_{\mf{D}^{*}}$ as $r \to 0+$. It follows that $\kappa_{c_{A_r}}$ converges uniformly on compact subsets of $\mf{D}^{*}$ to $-4$ as $r \to 0+$. Since $c_{A_r}(z)=2\pi S_{A_r}(z)$, it also follows that $S_{A_r}(z)$ converges uniformly on compact subsets of $\mf{D}^{*}$ to $S_{\mf{D}}\vert_{\mf{D}^{*}}(z)$ as $r \to 0+$. Accordingly, $s_{A_r}$ and $\kappa_{s_{A_r}}$ converge uniformly on compact subsets of $\mf{D}^{*}$ to $s_{\mf{D}}\vert_{\mf{D}^{*}}$ and $-4$ respectively as $r \to 0+$. This does not contradict Theorem~\ref{variation} as in this theorem, the points where the limits are studied, \textbf{do not} lie on a fixed compact subset of $\mf{D}^{*}$.

Second, using Theorem~\ref{variation}, we can answer the questions related to the universal bounds of the Gaussian curvatures of the Carath\'eodory and Szeg\"o metrics. Observe from (d) that $\kappa_{s_{A_r}}(r^{4/9}) \to 4$ as $r \to 0+$. This establishes the following:

\begin{thm}
Given $\epsilon>0$, there exists a $C^{\infty}$-smoothly bounded domain $\Omega_{\epsilon} \subset \mf{C}$ and a point $z\in \Omega_{\epsilon}$ such that $\kappa_{s_{\Om_\ep}}(z)>4-\ep$.
\end{thm}

Also, from (c) and (d), $\kappa_{c_{A_r}}(r^{1/2}) \to -\infty$ and $\kappa_{s_{A_r}}(r^{1/4}) \to -\infty$ as $r \to 0+$. Thus we obtain

\begin{thm}
There are no universal lower bounds for the Gaussian curvatures of the Carath\'eodory and Szeg\"o metrics on
the class of $C^{\infty}$-smoothly bounded planar domains.
\end{thm}

Third, Theorem~\ref{variation} also shows that there are domains in $\mf{C}$ such that the Gaussian curvature of the Szeg\"o metric assume both negative and positive values:

\begin{thm}
There exists a $C^{\infty}$-smoothly bounded domain $\Omega\subset \mf{C}$ and $z, w\in \Omega$, such that $\kappa_{s_{\Om}}(z)>0$ and $\kappa_{s_\Omega}(w)<0$.
\end{thm}

Lastly, Theorem~\ref{variation} also allows us to answer a question related to the comparison of the Carath\'eodory and Szeg\"o metrics. It was shown in \cite{bl14} that $s_{\Om} \geq c_{\Om}$, and therefore the ratio $s_{\Om}/c_{\Om}$ has the universal lower bound $1$. On the other hand, since $s_{A_r}(r^{1/2})/ c_{A_r}(r^{1/2}) \to +\infty$ as $ r\to 0+$, we have
\begin{thm}
There is no universal constant $M>0$ such that
\[
s_{\Om}(z)/c_{\Om}(z) \leq M,
\]
for all $z\in \Om$ and for all $C^{\infty}$-smoothly bounded domains $\Om \subset \mf{C}$.
\end{thm}

\section{Some observations}
In this section, we present some auxiliary results on the Szeg\"o kernel and the Szeg\"o metric that are needed to prove our theorems.
\subsection{Localisation of $S_{\Om}(z)$}
The following localisation of the Carath\'eodory metric is well-known (see \cite[Section 19.3]{jp13}): if $\Om \subset \mf{C}$ is a $C^{\infty}$-smoothly bounded domain, $p \in b\Om$, and $U$ is a (sufficiently small) neighbourhood of $p$, then
\begin{equation}\label{loc-ca}
\lim_{z \to p} \frac{c_{U\cap \Om}(z)}{c_{\Om}(z)}=1.
\end{equation}
Now let $\ti \Om\subset \Om$ be $C^{\infty}$-smoothly bounded domains that share an open piece $\Ga \subset b\Om$. Let $p \in \Ga$ and choose a neighbourhood $U$ of $p$ sufficiently small so that $U \cap \ti \Om = U \cap \Om$ and \eqref{loc-ca} holds. Since the Carath\'eodory metric is distance decreasing,
\[
1 \leq \frac{c_{\ti \Om}(z)}{c_{\Om}(z)} \leq \frac{c_{U \cap \Om}(z)}{c_{\Om}(z)},
\]
and hence by \eqref{loc-ca},
\[
\lim_{z \to p} \frac{c_{\ti \Om}(z)}{c_{\Om}(z)}=1.
\]
Combining this with the identity $c_{\Om}=2\pi S_{\Om}$ on $C^{\infty}$-smoothly bounded domains, we have the following:
\begin{prop}\label{loc-S}
Let $\ti \Om \subset \Om \subset \mf{C}$ be $C^{\infty}$-smoothly bounded domains such that $b\ti \Om$ and $b\Om$ share an open piece $\Ga \subset b\Om$. Then for every $p \in \Ga$,
\[
\lim_{z \to p} \frac{S_{\ti \Om}(z)}{S_{\Om}(z)}=1.
\]
\end{prop}

\subsection{Scaling method}
Let $\Om \subset \mf{C}$ be a $C^{\infty}$-smoothly bounded domain and $p \in b\Om$. Let $\psi$ be a $C^{\infty}$-smooth local defining function for $\Om$ at $p$ defined on a neighbourhood $U$ of $p$. Let $(p_{j})_{j\geq 1}$ be a sequence in $U\cap \Om$ converging to $p$. Consider the affine maps $T_j: \mf{C} \to \mf{C}$ defined by
\begin{equation}\label{T_j}
T_{j}(z)=\frac{z-p_{j}}{-\psi(p_{j})},
\end{equation}
and let $\Om_j=T_j(\Om)$. Observe that $T_j(p_j)=0$, and thus every $\Om_j$ contains $0$. Moreover, the function
\[
\psi_j(z)=\frac{1}{-\psi(p_j)}\psi \circ T_j^{-1}(z),
\]
is a $C^{\infty}$-smooth local defining function for $\Om_j$ at $T_j(p)$ defined on $T_j(U)$. Observe that if $K$ is a compact subset of $\mf{C}$ then $K \subset T_j(U)$ for $j$ large and thus $\psi_j$ is defined on $K$. Moreover, for $z \in K$,
\begin{align*}
\psi_j(z) & =\frac{1}{-\psi(p_j)}\psi\big(p_j+z(-\psi(p_j)\big)\\
& = -1 + 2 \Re \big(\pa \psi(p_j) z\big) + \psi(p_j)o(1),
\end{align*}
by expanding $\psi$ in a Taylor series near $p_j$. Therefore, $(\psi_j)_{j \geq 1}$ converges uniformly on compact subsets of $\mf{C}$ to
\[
\psi_{\infty}(z) = -1 + 2 \Re \big(\pa\psi(p)z\big).
\]
As a consequence, a set is compactly contained in the half-plane
\begin{equation}\label{defn-half-plane}
\mathcal{H}=\Big\{z \in \mf{C}: -1 + 2 \Re \big(\pa\psi(p)z\big)<0\Big\}
\end{equation}
if and only if it is uniformly compactly contained in $T_j(U\cap \Om) \subset \Om_j$ for $j$ large. In other words, the sequences of domains $(\Om_j)_{j \geq 1}$ and $(T_j(U\cap \Om))_{j \geq 1}$ converge in the local Hausdorff sense to the half-plane $\mathcal{H}$. Moreover, if $\ti \Om \subset \Om$ is a $C^{\infty}$-smoothly bounded domain such that $ b\ti \Om$ and $b \Om$ share a neighbourhood of $p$ in $b \Om$, and $\ti \Om_j=T_j (\ti \Om)$, then taking $U$ sufficiently small, $T_j(U\cap \Om) \subset \ti \Om_j \subset \Om_j$, and thus the sequence of domains $(\ti \Om_j)_{j \geq 1}$ also converges in the local Hausdorff sense to $\mathcal{H}$.

To this end, we write down the Szeg\"o kernel for the half-plane $\mathcal{H}$. First, recall that the Szeg\"o kernel for the unit disc $\mf{D}$ is given by
\[
S_{\mf{D}}(z,\zeta)=\frac{1}{2\pi}\frac{1}{1-z \ov \zeta}.
\]
Since $z \mapsto (z-i)/(z+i)$ is a conformal equivalence of the upper half-plane $\mf{H}=\{z \in \mf{C} : \Im z>0\}$ onto $\mf{D}$, we have
\[
S_{\mathcal{H}}(z,\zeta)=\frac{i}{2\pi(z-\ov \zeta)}.
\]
Now, the conformal equivalence $z \mapsto -i(\pa\psi(p) z - 1/2)$ of  $\mathcal{H}$ onto $\mf{H}$ gives
\begin{equation}\label{szego-kernel-H}
S_{\mathcal{H}}(z,\zeta)=\frac{\big\vert \pa\psi(p)\big\vert}{2\pi\big(1- \pa\psi(p) z - \ov\pa\psi(p) \ov \zeta\big)}.
\end{equation}
It follows that
\begin{equation}\label{szego-metric-H}
s_{\mathcal{H}}(z)=\frac{\big \vert \pa \psi (p)\big\vert}{1- 2 \Re\big(\pa\psi(p) z\big)},
\end{equation}
and
\begin{equation}\label{h-ord-curv-hp}
\kappa_{m_{\mathcal{H}}}^{(N)}(z)=-4\left(\prod_{m=1}^{N}m! \right)^2,
\end{equation}
for all $z \in \mathcal{H}$.
\begin{prop}\label{stab-S_j}
The sequence $(S_{\Om_j}(z,\zeta))_{j \geq 1}$ converges uniformly on compact subsets of $\mathcal{H} \times \mathcal{H}$ to $S_{\mathcal{H}}(z,\zeta)$. Moreover, all the partial derivatives of $S_{\Om_j}(z,\zeta)$ converge to the corresponding partial derivatives of $S_{\mathcal{H}}(z,\zeta)$ uniformly on compact subsets of $\mathcal{H} \times \mathcal{H}$.
\end{prop}
This follows from Proposition~3.2 of \cite{sv20} together with the following observation that we add for clarity: choose $U$ above to be sufficiently small so that $U \cap \Om$ is simply connected. Fix $a \in \mathcal{H}$. Then $a \in T_j (U\cap \Om)$ for $j$ large. Since $T_j(U\cap \Om)$ is simply connected, by the proof of Proposition~3.1 in \cite{sv20},
\[
c_{T_j(U\cap \Om)}(a) \to c_{\mathcal{H}}(a).
\]
Also, by the transformation formula for the Carath\'eodory metric and by \eqref{loc-ca},
\[
\frac{c_{T_j(U\cap \Om)}(a)}{c_{\Om_j}(a)} =\frac{c_{U\cap \Om}\big(-\psi(p_j)a+p_j\big)}{c_{\Om}\big(-\psi(p_j)a+p_j\big)} \to 1.
\]
This implies that
\[
c_{\Om_j}(a) \to c_{\mathcal{H}}(a).
\]
Now, the proof of Proposition~3.2 of \cite{sv20} applies to show that $(S_{\Om_j}(z, \zeta))_{j \geq 1}$ converges uniformly on compact subsets of $\mathcal{H} \times \mathcal{H}$ to $S_{\mathcal{H}}(z, \zeta)$. The uniform convergence of the derivatives follows from the fact that these functions are holomorphic in $z, \ov \zeta$. Finally, since $S_{\Om_j}(z)$ and $S_{\mathcal{H}}(z)$ are nonvanishing, it is immediate from the above proposition that

\begin{cor}\label{stab-S_j-g_j}
All the partial derivatives of $S_{\Om_j}(z)$ and $s_{\Om_j}(z)$ converge uniformly on compact subsets of $\mathcal{H}$ to the corresponding derivatives of $S_{\mathcal{H}}(z)$ and $s_{\mathcal{H}}(z)$.
\end{cor}

\subsection{Boundary asymptotics}
\begin{prop}\label{asymp}
Let $\Om \subset \mf{C}$ be a $C^{\infty}$-smoothly bounded domain and $p \in b\Om$. Let $\psi$ be a $C^{\infty}$-smooth local defining function for $\Om$ at $p$ defined in a neighbourhood $U$ of $p$ and $\mathcal{H}$ be the half-plane defined by
\[
\mathcal{H}=\Big\{z \in \mf{C}: -1 + 2 \Re \big(\pa\psi(p)z\big)<0\Big\}.
\]
Then as $U\cap \Om \ni z \to p$,
\begin{itemize}
\item [(i)] $\pa^{k\ov l}S_{\Om}(z) \big(-\psi(z)\big)^{k+l+1}\to \pa^{k\ov l}S_{\mathcal{H}}(0)=\frac{(k+l)!}{2\pi} \big\vert \pa \psi(p)\big\vert \big(\pa \psi(p)\big)^k \big(\ov\pa\psi(p)\big)^l$.
\item [(ii)] $\pa^{k\ov l}s_{\Om}(z) \big(-\psi(z)\big)^{k+l+1} \to \pa^{k\ov l}s_{\mathcal{H}}(0)=(k+l)! \big\vert \pa \psi(p)\big\vert \big(\pa \psi(p)\big)^k \big(\ov\pa\psi(p)\big)^l$.
\end{itemize}
\end{prop}
\begin{proof}
Let $(p_{j})_{j \geq 1}$ be a sequence in $\Om$ such that $p_{j}\to p$. Then $p_j \in U$ for $j$ large. Let $\Om_j=T_j(\Om)$ where $T_j$ is as in \eqref{T_j}. Differentiating
\[
S_{\Om}(z)=S_{\Om_j}\big(T_j(z)\big)\big(-\psi(p_j)\big)^{-1},
\]
we obtain
\[
\pa^{k\ov l} S_{\Om}(p_j)\big(-\psi(p_j)\big)^{k+l+1} =\pa^{k\ov l}S_{\Om_j}(0) \to \pa^{k\ov l}S_{\mathcal{H}}(0),
\]
by Proposition~\ref{stab-S_j}, which establishes the first part of (i). The second part of (i) follows from an explicit calculation of the derivatives of $S_{\mathcal{H}}(z)$.

Similarly, (ii) is obtained by differentiating
\[
s_{\Om}(z)=s_{\Om_j}\big(T_j(z)\big) \big(-\psi(p_j)\big)^{-1},
\]
which completes the proof of the proposition.
\end{proof}

\subsection{Localisation of $\pa^{k \ov l} S_{\Om}$ and $\pa^{k \ov l} s_{\Om}$}
As an immediate consequence of the boundary asymptotics, we have the following localisation result that generalises Proposition~\ref{loc-S}:
\begin{prop}
Let $\ti \Om \subset \Om \subset \mf{C}$ be $C^{\infty}$-smoothly bounded domains such that $b\ti \Om$ and $b\Om$ share an open piece $\Ga \subset b\Om$. Then for $p \in \Ga$ and $k,l \geq 0$,
\[
\frac{\pa^{k\ov l}S_{\ti \Om}(z)}{\pa^{k\ov l}S_{\Om}(z)} \to 1 \quad \text{and} \quad \frac{\pa^{k\ov l}s_{\ti \Om}(z)}{\pa^{k\ov l}s_{\Om}(z)} \to 1,
\]
as $z \to p$, $z \in \ti \Om$.
\end{prop}
\begin{proof}
Choose a common $C^{\infty}$-smooth local defining function $\psi$ for $\Om$ and $\ti \Om$ at $p$ defined on a neighbourhood $U$ of $p$.  Then, by Proposition~\ref{asymp}, both $\pa^{k \ov l}S_{\Om}(z)$ and $\pa^{k\ov l}S_{\ti \Om}(z)$ (resp. $\pa^{k\ov l}s_{\Om}(z)$ and $\pa^{k\ov l}s_{\ti \Om}(z)$) have the same boundary asymptotics and hence the proposition follows.
\end{proof}

\subsection{Comparison with the classical metrics}
Another consequence of Proposition~\ref{asymp} is that the Szeg\"o metric is comparable with the hyperbolic metric. Let $\rho_{\Om}(z)\vert dz\vert$ be the hyperbolic metric and $s_\Om(z)\vert dz\vert$ the Szeg\"o metric on a $C^{\infty}$-smoothly bounded domain $\Om \subset \mf{C}$. By Proposition~\ref{asymp} (ii), there exists a constant $C>1$ such that for every $z\in \Om$,
\begin{equation}\label{g1}
 C^{-1}\rho_{\Om}(z)\leq s_{\Om}(z)\leq C\rho_{\Om}(z).
\end{equation}
It follows that $s_{\Om}(z)\vert dz\vert$ is comparable to $\rho_{\Om}(z) \vert dz\vert$ and hence also to the Carath\'eodory and the Bergman metric.

The above observation implies that the Szeg\"o metric on $\Om$ is Gromov hyperbolic. In a coarse sense, a Gromov hyperbolic metric space behaves like a negatively curved manifold and is defined as follows. Let $(X, d)$ be a metric space and $x,y \in X$. By a geodesic segment in $X$ joining $x$ and $y$, we mean continuous map $\sigma:[a,b] \to X$, where $[a,b] \subset \mf{R}$ is a closed interval, such that $\sigma(a)=x$, $\sigma(b)=y$, and for every $s,t \in [a,b]$,
\[
d\big(\sigma(s),\sigma(t)\big)=\vert s-t\vert.
\]
A geodesic segment joining $x$ and $y$, despite its possible non-uniqueness, will be denoted by $[x, y]$. The space $(X, d)$ is called a geodesic space if for every pair of points $x,y \in X$, there is a geodesic segment joining $x$ and $y$. Given $\de> 0$, a geodesic metric space $(X, d)$ is called $\delta$-hyperbolic if every geodesic triangle $[x,y] \cup [y,z] \cup [z,w]$ in $X$ is $\delta$-thin, i.e., 
\[
d\big(w, [y, z]\cup [z, x]\big)\leq \delta,
\]
for all $w\in [x,y]$. The metric space $(X,d)$ is called Gromov hyperbolic if there exists a $\de>0$ such that $(X,d)$ is $\de$-hyperbolic.

For brevity, we will denote the distance functions induced by $\rho_{\Om}$ and $s_{\Om}$ by the same notations $\rho_{\Om}$ and $s_{\Om}$, respectively.
\begin{cor}
Let $\Om \subset \mf{C}$ be a non-degenerate finitely connected domain. Then the metric space $(\Om, s_{\Om})$ is Gromov hyperbolic.
\end{cor}
\begin{proof}
Without loss of generality, we may assume that $\Om$ is $C^{\infty}$-smoothly bounded. Then, from \cites{rt4}, $(\Om,\rho_{\Om})$ is $\delta$-hyperbolic for some $\delta> 0$. Since $\rho_{\Om}(z)|dz|$ is complete, this implies that $s_{\Om}(z)\vert z \vert$ is also complete and hence $(\Om, s_{\Om})$ is a geodesic space. Moreover, the identity map between $(\Om,\rho_{\Om})$ and $(\Om, s_{\Om})$ is a quasi-isometry and consequently $(\Om,s_{\Om})$ is also $\delta$-hyperbolic, possibly for a different choice of $\delta$.
\end{proof}

\section{Geodesics of the Carath\'eodory and Szeg\"o metrics}
In this section, we prove Theorem~\ref{geo}.
\subsection{Existence of closed geodesic}
We recall the following result of Herbort:

\begin{thm}[Theorem 1.1 \cite{hg83} ]
Let $G$ be a bounded domain in $\mf{R}^{N}$, where $N \in \mf{N}$, such that $\pi_1(G)$ is nontrivial. Assume that the following conditions are satisfied:
\begin{itemize}
\item [(i)] For each $p \in \ov G$, there is an open neighbourhood $U\subset\mf{R}^N$ such that the set $G \cap U$ is simply connected.
\item [(ii)] $G$ is equipped with a complete Riemannian metric $g$ which possesses the following property: (B) For each $S>0$ there exists a $\de>0$, such that for each $p \in G$ with $d(p, bG)<\de$ and every $X \in \mf{R}^N$, $g(p,X) \geq S\|X\|^2$ (where $\|\cdot \|$ denotes the Euclidean norm). 
\end{itemize}
Then every nontrivial homotopy class in $\pi_1(G)$ contains a closed geodesic for $g$.
\end{thm}

\begin{proof}[Proof of Theorem \ref{geo} (i)]
Without loss of generality, assume that $\Om$ is $C^{\infty}$-smoothly bounded. Also, it is enough to show that $ds_{m_\Om}^2=m^2_{\Om}(z) \vert dz\vert^2$ satisfies the hypothesis of the above theorem. Let $\psi$ be a $C^{\infty}$-smooth defining function for $\Om$. By Proposition~\ref{asymp} (ii), for $z \in D$ and $v \in \mf{C}$, we have
\[
\frac{ds_{m_{\Om}}^{2}(z,v)}{|v|^{2}}=m_{\Om}^2(z)\gtrsim \big(\psi(z)\big)^{-2},
\]
which implies that the Property (B) is satisfied. All the other conditions are evidently true. This proves (i).
\end{proof}

\subsection{Existence of geodesic spirals} 
The main tool for the proof of Theorem~1.1 (ii) is a result of Herbort from \cite{h83}. To state this result we first recall the notion of a geodesic loop. Let $(M,g)$ be a complete Riemannian manifold.
    \begin{itemize}
        \item [(a)] A \textit{geodesic spiral} is a non-closed geodesic $\sigma:\mathbb{R}\to M$ in which each point $\sigma(t)$ for $t\geq 0$ belongs to some compact subset $K$ of $M$.
          \item [(b)] For a nontrivial geodesic $\sigma:\mathbb{R}\to M$, if there exist $t_{1},~t_{2}\in\mathbb{R}$ with $\sigma(t_{1})=\sigma(t_{2})$, then the segment $\sigma_{|_{[t_{1},t_{2}]}}$ is referred to as a geodesic loop passing through $\sigma(t_{1})$.
    \end{itemize}

\begin{lem}[Lemma 2.2, \cite{h83}]\label{Lem-Her-cpt}
Let $(M,g)$ be a complete Riemannian manifold the universal cover of which possesses infinitely many leaves. Let $x_0$ be a point in $M$ such that there is no closed geodesic passing through $x_0$. Assume that there exists a compact subset $K$ of $M$ with the property that each geodesic loop passing through $x_0$ is contained in $K$. Then there exists a geodesic spiral passing through $x_0$.
\end{lem}

In view of this lemma, the problem of showing the existence of geodesic spirals for $m_D(z)\vert dz\vert$ now reduces to finding an appropriate compact subset $K \subset \Om$. For this, we require the following:

\begin{prop}\label{spiral}
Let $\Om \subset \mf{C}$ be a $C^{\infty}$-smoothly bounded domain with a $C^{\infty}$-smooth defining function $\psi$. Then there exists $\ep>0$ such that for each geodesic $\sigma:\mathbb{R}\to \Om$ of the metric $m_\Om(z) \vert dz\vert$ satisfying $(\psi\circ \sigma)(0)>-\epsilon$ and $(\psi\circ \sigma)'(0)=0$, we have $(\psi\circ \sigma)''(0)>0$.   
\end{prop}

\begin{proof}
Suppose to the contrary, for each $k\in\mathbb{N}$, there exists a geodesic $\sigma_{k}$ of $m_{\Om}(z) \vert dz\vert$ satisfying
\[
\text{(a)} \, \psi\big(\sigma_{k}(0)\big)> -\frac{1}{k},\quad \text{(b)} \,(\psi\circ \sigma_{k})'(0)=0, \quad \text{and} \quad \text{(c)} \,(\psi\circ \sigma_{k})''(0)\leq 0.
\]
Let us denote for each $k$,
\[
p_{k}=\sigma_{k}(0), \quad v_{k}=\frac{\sigma_{k}'(0)}{\big\vert \sigma_{k}'(0)\big\vert}, \quad \text{and} \quad b_{k}=\frac{(\psi\circ \sigma_{k})''(0)}{\big \vert \sigma_{k}'(0)\big \vert^{2}}\leq 0.
\]
By passing to the subsequence, let  $p_{k}\to p_{0}$ and $v_{k}\to v_{0}$ where $|v_{0}|=1$. By a translation and rotation of $D$, we may assume that
\[
p_{0}=0\text{ and }\frac{\partial\psi(0)}{\partial z}=1.
\]
Then from (b) we have
\begin{equation}\label{re-v_0}
 0=\lim_{k\to\infty}\text{Re}\left(\frac{\pa\psi(p_{k})}{\pa z}v_{k}\right)=\text{Re}\left(\frac{\partial \psi(0)}{\partial z}v_{0}\right)= \text{Re}(v_{0}).
 \end{equation}
On the other hand, from (c) we have
\begin{align}\label{s1}
\text{Re}\left(\frac{\partial \psi(p_{k})}{\partial z}\sigma_{k}''(0)\right)+\text{Re}\left(\frac{\partial^{2}\psi(p_{k})}{\partial z^{2}}\big(\sigma_{k}'(0)\big)^{2}\right)+\frac{\partial^{2}\psi(p_{k})}{\partial z\partial\ov{z}}\big \vert \sigma_{k}'(0)\big \vert^{2} \leq 0.
\end{align}
Since $\sigma_k$ is a geodesic of the metric $m_{\Om}^2 \vert dz\vert^2$, we have
\begin{equation}\label{geo-eq}
-\sigma_{k}''=\frac{1}{m_{\Om}^2(\sigma_{k})}\frac{\partial m_{\Om}^2(\sigma_{k})}{\partial z}(\sigma_{k}')^{2}.
\end{equation}
Thus, combining \eqref{s1} and \eqref{geo-eq}, we obtain
\[
-\Re\left(\frac{\pa \psi}{\pa z}(p_k)\frac{1}{m_{\Om}^2(p_k)}\frac{\partial m_{\Om}^2(p_{k})}{\partial z}\big(\sigma_{k}'(0)\big)^{2}\right) + \text{Re}\left(\frac{\partial^{2}\psi(p_{k})}{\partial z^{2}}\big(\sigma_{k}'(0)\big)^{2}\right)+\frac{\partial^{2}\psi(p_{k})}{\partial z\partial\ov{z}}\big \vert \sigma_{k}'(0)\big \vert^{2} \leq 0.
\]
Dividing throughout by $|\sigma_{k}'(0)|^{2}$ and multiplying by $-\psi(p_{k})$, we have
\begin{multline}\label{2nd-der-eq}
-\Re\left(\frac{\pa \psi}{\pa z}(p_k)\big(-\psi(p_k)\big)\frac{1}{m_{\Om}^2(p_k)}\frac{\partial m_{\Om}^2(p_{k})}{\partial z}v_k^{2}\right) + \big(-\psi(p_k)\big)\text{Re}\left(\frac{\partial^{2}\psi(p_{k})}{\partial z^{2}}\big(\sigma_{k}'(0)\big)^{2}\right)\\+\big(-\psi(p_k)\big)\frac{\partial^{2}\psi\big(p_{k}\big)}{\partial z\partial\ov{z}}\big \vert \sigma_{k}'(0)\big \vert^{2} \leq 0.
\end{multline}
Observe that the last two terms converge to $0$ as $k \to \infty$. On the other hand, by Proposition~\ref{asymp},
\[
\big(-\psi(p_k)\big)\frac{1}{m_{\Om}^2(p_k)}\frac{\partial m_{\Om}^2(p_{k})}{\partial z} = \frac{1}{m_{\Om}^2(p_k)(-\psi(p_k))^2}\frac{\partial m_{\Om}^2(p_{k})}{\partial z}(-\psi(p_k))^3 \to 2.
\]
By taking $k \to \infty$ in \eqref{2nd-der-eq}, it follows that
\[
\Re\left(v_0^2\right) \geq 0.
\]
which is a contradiction as $v_0 =\pm i$ from \eqref{re-v_0}.
 \end{proof}

We are now in a position to give a proof of Theorem~1.1 (ii).

\begin{proof}[Proof of Theorem~1.1 (ii)]
Without loss of generality, assume that $\Om$ is $C^{\infty}$-smoothly bounded. Let $z_{0}\in \Om$ be such that no closed geodesics passes through it. Let $\psi$ be a $C^{\infty}$-smooth defining function for $\Om$ and let $\ep>0$ be as in Proposition~\ref{spiral}. Let $\epsilon_{1}=\min\{\epsilon,~-\psi(z_{0})\}$. Set
\[
K=\big\{z\in \Om: \psi(z)\leq -\epsilon_{1} \big\}.
\]
Then the compact set $K$ has the property as in Lemma~\ref{Lem-Her-cpt}. Indeed, let $\sigma |_{[t_{1},t_{2}]}: [t_{1},t_{2}]\to \Om$ be a geodesic loop that passes through $z_{0}$ and suppose that $\sigma |_{[t_{1},t_{2}]}([t_{1},t_{2}])\not\subset K$. Since $(\psi\circ \sigma)|_{[t_{1},t_{2}]}$ is a continuous real-valued function, it will attain maximum at some point $t_{0}\in (t_{1}, t_{2})$. From the definition of $K$ and the fact $\sigma|_{[t_{1},t_{2}]}$ leaves $K$ implies that $\sigma(t_{0})\in \Om\setminus {K}$, which further implies that $(\psi\circ \sigma)(t_{0})>-\epsilon,~(\psi\circ \sigma)'(t_{0})=0$, and $(\psi\circ \sigma)''(t_{0}) \leq 0$. But it contradicts the above proposition. By Lemma~\ref{Lem-Her-cpt}, the proof of (ii) follows.
\end{proof}

\section{Proof of Theorem~\ref{bdry curv}}
First, assume that $\Om$ is $C^{\infty}$-smoothly bounded and $p \in b\Om$.  By Proposition~\ref{asymp} and recalling that $c_{\Om}(z)=2\pi S_{\Om}(z)$, we have for $k,l\geq 0$,
\[
\pa^{k\ov{l}}m_{\Om}(z)\big(-\psi(z)\big)^{k+l+1} \to \pa^{k\ov{l}} m_{\mathcal{H}}(0),
\]
as $z \to p$. Therefore, denoting the symmetric group over $\{0, \ldots, N+1\}$ by $S_{N+1}$,
\begin{align*}
 \det \begin{pmatrix}\pa^{k\ov l} m_{\Om}(z)\end{pmatrix}_{k,l=0}^N\big(-\psi(z)\big)^{(N+1)^2}
 &=\sum_{\sigma\in S_{N+1}}\text{sign}(\sigma)\prod_{k=0}^{N}\left(\pa^{k\ov {\sigma(k)}} m_{\Om}(z) \big(-\psi(z)\big)^{k+\sigma(k)+1}\right)\\
 & \to \sum_{\sigma\in S_{N+1}}\text{sign}(\sigma)\prod_{k=0}^{N}\left(\pa^{k\ov {\sigma(k)}} m_{\mathcal{H}}(0) \right)\\
 & = \det \begin{pmatrix}\pa^{k\ov l} m_{\mathcal{H}}(0)\end{pmatrix}_{k,l=0}^N.
 \end{align*}
This implies that
\begin{align*}
\kappa_{\Om}^{(N)}(z)&=-4\frac{ \det \begin{pmatrix}\pa^{k\ov l} m_{\Om}(z)\end{pmatrix}_{j,k=0}^N\big(-\psi(z)\big)^{(N+1)^2}} {m_{\Om}^{(N+1)^2}(z)\big(-\psi(z)\big)^{(N+1)^2}}
\to -4\frac{\det \begin{pmatrix}\pa^{k\ov l} m_{\mathcal{H}}(0)\end{pmatrix}_{k,l=0}^N}{m^{(N+1)^2}_{\mathcal{H}}(0)} = \kappa_{m_{\mathcal{H}}}^{(N)}(0),
\end{align*}
as $z \to p$. Recall from \eqref{h-ord-curv-hp} that right-hand side is $-4(\prod_{m=1}^{N}m!)^2$ and the proof is complete when $\Om$ is $C^{\infty}$-smoothly bounded.

For the general case, let $\phi: \Om \to \Om'$ be a conformal equivalence where $\Om' \subset \mf{C}$ is a $C^{\infty}$-smoothly bounded domain. Let $p \in b \Om$ and $(z_j)_{j \geq 1}$ be a sequence in $\Om$ such that $z_j \to p$. Since the sequence $(\phi(z_j))_{j \geq 1}$ is bounded, it has a convergent subsequence, say $(\phi(z_{j_k}))_{k\geq 1}$. Note that the limit of this subsequence must lie in $b\Om'$. It follows from the previous case that
\[
\kappa_{m_{\Om}}^{(N)}(z_{j_k})=\kappa_{m_{\Om'}}^{(N)}\big(\phi(z_{j_k})\big) \to -4\left(\prod_{m=1}^{N}m! \right)^2.
\]
Thus, we have shown that every sequence in $\Om$ converging to $p$ admits a subsequence along which $\kappa_{m_\Om}^{(N)}$ converges to $-4(\prod_{m=1}^{N}m!)^2$. It follows that $\kappa_{m_{\Om}}^{(N)}(z) \to -4(\prod_{m=1}^{N}m! )^2$ as $z \to p$. This completes the proof of the theorem. \hfill \qed

\section{Proof of Theorem~\ref{l2-co}}
Without loss of generality, we may assume that $\Om$ is $C^{\infty}$-smoothly bounded. Also, we will prove the theorem for $m_{\Om}=s_{\Om}$ only as the case $m_{\Om}=c_{\Om}$ is similar. Fix a $C^{\infty}$-smooth defining function $\psi$ for $\Om$.

\textbf{Case $p+q \neq 1$.} Since $\mathcal{H}^0_2(\Om)\cong \mathcal{H}^2_2(\Om)$, it is enough to prove that every square integrable harmonic function on $\Om$ with respect to $ds_{s_{\Om}} = s_{\Om}(z) \vert dz \vert$ is identically equal to $0$. Let $f$ be such a function. First, note that $f$ must be constant since $ds$ is complete and K\"{a}hler (see for instance \cite{y76}). Also, by Proposition~\ref{asymp}, $\Om$ has infinite volume with respect to $ds_{s_{\Om}}$:
\[
\int_{\Om} \frac{i}{2} s^2_{\Om}(z) \, dz \wedge d\ov z \gtrsim \int_{\Om} \frac{1}{\big(\psi(z)\big)^2} dA(z) = +\infty,
\]
which implies that $f$ must be identically equal to $0$.

\textbf{Case $p+q=1$.} By \cite{o89}, the infinite dimensionality of $\mathcal{H}^{p,q}_2(\Om)$ will follow at once if we establish that
\begin{equation}\label{oh-eq}
ds_{s_{\Om}}^2 \approx (-\psi)^{-1} \vert dz\vert^2 + (-\psi)^{-2}\vert \pa \psi\vert^2 \vert dz\vert^2,
\end{equation}
uniformly near $b\Om$. If $p \in b \Om$ then by Proposition~\ref{asymp}, 
\[
\lim_{z \to p} \frac{s_{\Om}(z)^2}{\big(-\psi(z)\big)^{-1} + \big(-\psi(z)\big)^{-2} \big\vert \pa \psi(z)\big\vert^2}
= \lim_{z\to p} \frac{\big(-\psi(z)\big)^2 s^2_{\Om}(z)}{ -\psi(z) + \big\vert \pa \psi (z) \big\vert^2 }= \frac{\big\vert \pa \psi (p) \big\vert^2}{\big\vert \pa \psi (p) \big\vert^2 },
\]
which shows that \eqref{oh-eq} holds near $p$. By compactness of $b\Om$, it follows that \eqref{oh-eq} holds near $b\Om$ which completes the proof. \hfill \qed

\section{Proof of Proposition~\ref{sgo ann}}
We begin by recalling the series form of the Szeg\"o kernel on the annulus
\[
A_r=\big\{z \in \mf{C} : r< \vert z \vert <1\big\},
\]
where $r \in (0,1)$. An orthonormal basis for $H^2(bA_r)$ is given by
\[
\left\{\frac{z^n}{\sqrt{2\pi(1+r^{2n+1})}}\right\}_{n=-\infty}^{\infty},
\]
and hence the Szeg\"o kernel for $A_r$ is
\[
S_{A_r}(z,w)=\frac{1}{2\pi} \sum_{n=-\infty}^{\infty} \frac{(z\ov w)^n}{1+r^{2n+1}}.
\]
To find a closed form of the Szeg\"o metric, we first recall the definitions of the Weierstrass elliptic functions following the notations in \cite{a05} and \cite{m35}. For more details on the theory of elliptic functions, we refer to \cite{WW}. Let $\om_1= -\log r$, $\om_3=i\pi$, and write $\Om_{m,n}=2m\om_1+2n\om_3$ for $m,n \in \mf{Z}$. The Weierstrass elliptic $\wp$ function is defined by
\[
\wp(z)=\frac{1}{z^2}+{\sum_{m,n}}^\prime \left( \frac{1}{(z-\Om_{m,n})^{2}}-\frac{1}{\Om_{m,n}^{2}}\right),
\]
where the $'$ in the summation means that simultaneous zero values of $m,n$ are excepted. The function $\wp$ is holomorphic on $\mathbb{C}$ except for poles at $\Om_{m,n}$ for each $m,n \in \mf{Z}$. Moreover, it is an even function, it is doubly periodic with periods $2\om_1$ and $2\om_3$, and satisfies the differential equation
\[
(\wp')^2=4\wp^3-g_2\wp -g_3,
\]
where
\[
g_2=60 {\sum_{m,n}}'\frac{1}{\Om_{m,n}^4}, \quad g_3=140{\sum_{m,n}}'\frac{1}{\Om_{m,n}^6}.
\]
We also note that the roots of the equation $4x^3-g_2x-g_3$ are given by
\begin{equation}\label{e_i}
e_1=\wp(\om_1), \quad e_2=\wp(-\om_1-\om_3), \quad e_3=\wp(\om_3),
\end{equation}
which are real and satisfies $e_1>e_2>e_3$. Along the boundary of the half-period rectangle with vertices $0, \om_1, \om_1+\om_3, \om_3$, the function $\wp(z)$ is real, and as $z$ is taken counterclockwise around this rectangle from $0 \to \om_1 \to \om_1+\om_3\to \om_3 \to 0$, $\wp(u)$ decreases from $+\infty \to e_1\to e_2 \to e_3\to -\infty$. The behaviour of $\wp$ along the boundaries of the other half-period rectangles can be obtained from this using $p(-z)=p(z)$, $p(\ov z)=\ov p(z)$, and periodicity.

The Weierstrass $\zeta$ function is defined by
\begin{align*}
& \zeta'(z)=-\wp(z),\\
& \lim_{z \to 0} \big(\zeta(z)-z^{-1}\big)=0.
\end{align*}
The function $\zeta$ is odd and satisfies the quasi-periodicity condition
\[
\zeta(z +2 \om_k)=\zeta(z)+2\eta_k,
\]
for $k=1,3$, where $\eta_k=\zeta(\om_k)$.

Finally, the Weierstrass $\sigma$ function is defined by the equation
\begin{align*}
& \frac{d}{dz} \log \sigma(z)=\zeta(z),\\
&\lim_{z \to 0} \sigma(z)/z=1.
\end{align*}
The function $\sigma$ is odd and satisfies the quasi-periodicity condition
\[
\sigma(z+2\om_k)=-e^{2\eta_k(z+\om_k)}\sigma(z),
\]
for $k=1,3$. For real $u$, the functions $\wp(u)$, $\sigma(u)$, and $\zeta(u)$ take real values since $\omega_{1}$ is positive and $\omega_{2}$ is purely imaginary.

It is customary to write $\om_2=-\om_1-\om_3$, $\eta_2=-\eta_1-\eta_3$, and define three other $\sigma$ functions $\sigma_k(u)$, $k=1,2,3$, by the equation
\begin{equation}\label{sigma2}
\sigma_k(u)=e^{-\eta_k u}\frac{\sigma(u+\om_k)}{\sigma(\om_k)}.
\end{equation}

We also recall a relation between the Szeg\"o kernel and the capacity metric from \cite[Equation (7)]{a05}:
\begin{align}
2\pi S_{A_r}(z) & = c_{\beta}(z)\sigma_{2}^{*}\big(-2\log|z|\big)\label{s2}, 
\end{align}
where $c_{\beta}$ denotes the logarithmic capacity on $A_{r}$, and
\begin{equation}\label{sigma*}
{ \sigma_2^{*}}^2(u)=e^{-cu^2}\sigma_2^2(u),
\end{equation}
where $c=\eta_1/\om_1$.

\begin{proof}[Proof of Proposition~\ref{sgo ann}]
Differentiating \eqref{s2}, we obtain
\begin{equation*}
\partial\ov{\partial}\log S_{A_r}(z)=\partial\ov{\partial}\log c_{\beta}(z)+\partial\ov{\partial}\log\Big(\sigma_{2}^{*}\big(-2\ln|z|\big)\Big).
\end{equation*}
From, Suita \cites{s72}, we know that
\begin{equation}\label{capa}
\partial\ov{\partial}\big(\ln c_{\beta}(z)\big)=\pi K_{A_{r}}(z)=\frac{1}{|z|^{2}}\Big(\wp\big(2\ln \vert z\vert \big)+c\Big).
\end{equation}
The latter expression was originally found by Zarankiewicz \cite{Zara34}. On the other hand, writing $u=-2\log \vert z \vert$, which is harmonic, we obtain from \eqref{sigma*}
and \eqref{sigma2}, that
\[
\pa \ov {\pa} \log \big({\sigma_2^{*}}^2(u)\big)= \pa \ov {\pa}(-cu^2)  + \pa \ov \pa \log \sigma^2(u+\om_2).
\]
Note that
\[
\pa \ov \pa (-cu^2)=-c \pa (2u\ov \pa u)=-2c \vert \pa u\vert^2 =-\frac{2c}{\vert z \vert^2},
\]
and
\begin{align*}
\pa \ov \pa \log \sigma^2(u+\om_2)&=2\pa \left(\frac{d}{du}\big(\log \sigma(u+\om_2)\big)\ov \pa u\right)\\
&= 2\pa \big( \zeta(u+\om_2)\ov \pa u\big)\\
&= 2 \frac{d}{du} \big(\zeta(u+\om_2)\big) \vert \pa u \vert^2\\
& = -2\wp(u+\om_2)\frac{1}{\vert z \vert^2}.
\end{align*}
Thus,
\begin{equation}\label{sig*}
\partial\ov{\partial}\log\Big(\sigma_{2}^{*}\big(-2\log|z|\big)\Big) = -\frac{\wp\big(-2\log \vert z \vert+\om_2\big)+c}{\vert z \vert^2}=-\frac{\wp\big(2\log \vert z \vert+\om_1+\om_3\big)+c}{\vert z \vert^2}.
\end{equation}
Adding \eqref{capa} and \eqref{sig*}, we obtain
\[
\pa \ov \pa \log S_{A_r}(z)=\frac{\wp\big(2\log \vert z \vert\big)-\wp\big(2\log \vert z \vert+\om_1+\om_3\big) }{\vert z \vert^2},
\]
as required.
\end{proof}

\section{Variations on planar annuli}
In this section, we shall present the proof of Theorem~\ref{variation}. The idea is to use the Bergman-Fuks type formulas for the Szeg\"o kernel and the metric that relates them to certain maximal domain functions. Let us recall their definitions first. For a $C^{\infty}$-smoothly bounded domain $\Om \subset \mf{C}$, $j=0,1,2, \cdots$, and $z \in \Om$, we define the maximal domains functions $J_{\Om}^{(j)}$ by
\[
J_{\Om}^{(j)}(z)=\sup\Big\{\big \vert f^{(j)} (z)\big \vert^{2}: f\in H^{2}(b\Om), f(z)=f'(z)=\cdots = f^{(j-1)}(z)=0, \|f\|_{H^{2}(b\Om)}\leq 1\Big\}.
\]
It can be shown that
\begin{equation}\label{berg-fuks-formula}
\begin{aligned}
S_{\Om}(z) & =J^{(0)}(z),\\
s_{\Om}(z) & =\sqrt{\frac{J_{\Om}^{(1)}(z)}{J_{\Om}^{(0)}(z)}},\\
\kappa_{s_\Om(z)} & =4-2\frac{J_{\Om}^{(0)}(z)J_{\Om}^{(2)}(z)}{J_{\Om}^{(1)}(z)^2}.
\end{aligned}
\end{equation}
These formulas can be derived in the same way as for the usual Bergman kernel and the metric (see for example, \cite{Ber33}, \cite{bg70}, \cite{Fuks}, and \cite{jp13}), and so we do not repeat them here. Observe from the last formula that the Gaussian curvature of the Szeg\"o metric is at most $4$ on a $C^{\infty}$-smoothly bounded domain and hence on any non-degenerate finitely connected domain. Under a conformal equivalence $\phi: \Om \to \Om'$, the maximal domain functions $J_{\Om}^{(j)}(z)$, $j=0,1,2, \ldots$, transform according to the following rule
\begin{equation}\label{tr Js}
J_{\Om}^{(j)}(z)=\big \vert \phi'(z)\big \vert^{2j+1}J_{\Om'}^{(j)}\big(\phi(z)\big).
\end{equation}
Because of \eqref{berg-fuks-formula}, the focus now is to investigate the maximal domain functions on $A_r$. As a first step, we compute these domain functions on a general annulus
\[
A(r, R)=\big\{z\in\mf{C}: r<|z|<R\big\},
\]
where $0<r<1<R<\infty$. It is apparent from their definition that we will require suitable functions in $H^2(bA(r,R))$ to compute them and so first we prove a result that allows us to construct plenty of them. Note that $(z^n)_{n=-\infty}^{\infty}$ is an orthogonal basis for $H^2(bA(r,R))$ and writing $\al_{n}^{r,R}=\|z^n\|_{H^2(bA(r,R))}^2$, we have
\begin{equation}\label{H-norm-z^n}
\al_{n}^{r,R}= \int_{\vert z \vert =r}\vert z \vert^{2n} \vert dz\vert +\int_{\vert z \vert=R} \vert z \vert^{2n} \vert dz\vert = 2\pi\big(r^{2n+1}+R^{2n+1}\big).
\end{equation}
\begin{prop}\label{al_n-summable}
Let $p$ be a monic polynomial and
\[
c_n= \frac{p(n)}{\al_n^{r,R}}, \quad n \in \mf{Z}.
\]
Then $(c_n)_{-\infty}^{\infty}$ is summable, and hence also square summable. In particular,
\[
f(z)=\sum_{n=-\infty}^{\infty} c_n z^n,
\]
is in $H^2\big(bA(r,R)\big)$.
\end{prop}
\begin{proof}
Let $n\geq0$ be large. Then $p(\pm n) \neq 0$ and
\[
\frac{\vert p(n+1)\vert}{\vert p(n)\vert} \to 1 \quad \text{and} \quad \frac{\vert p(-n-1)\vert}{\vert p(-n)\vert} \to 1.
\]
Now,
\begin{align*}
\left \vert \frac{c_{n+1}}{c_n} \right \vert & =\frac{\vert p(n+1)\vert }{\vert p(n)\vert}\frac{\al_n^{r,R}}{\al_{n+1}^{r,R}}\\
& = \frac{\vert p(n+1)\vert }{\vert p(n)\vert}\frac{r^{2n+1}+R^{2n+1}}{r^{2n+3}+R^{2n+3}}\\
& = \frac{\vert p(n+1)\vert }{\vert p(n)\vert}\frac{R^{2n+1}\big((r/R)^{2n+1} +1\big)}{R^{2n+3}\big((r/R)^{2n+3}+1\big)}\\
& \to 1/R^2<1.
\end{align*}
Similarly,
\begin{align*}
\left \vert \frac{c_{-n-1}}{c_{-n}} \right\vert & =\frac{\vert p(-n-1)\vert}{\vert p(-n)\vert} \frac{\al_{-n}^{r,R}}{\al_{-n-1}^{r,R}}\\
&  = \frac{\vert p(-n-1)\vert}{\vert p(-n)\vert} \frac{r^{-2n+1}+R^{-2n+1}}{r^{-2n-1}+R^{-2n-1}} \\
& =\frac{\vert p(-n-1)\vert}{\vert p(-n)\vert} \frac{r^{-2n+1}\big(1+(R/r)^{-2n+1}\big)}{r^{-2n-1}\big(1+(R/r)^{-2n-1}\big)}\\
& \to r^2<1.
\end{align*}
It follows from the ratio test that $(c_n)_{n=-\infty}^{\infty}$ is summable.
\end{proof}
Now, to compute the maximal domain functions on $A(r,R)$, let us consider the series
\begin{equation}\label{phi_k}
s^{r,R}_k=\sum_{n=-\infty}^{\infty} \frac{n^k}{\al_n^{r,R}}, \quad k=0,1,2, \ldots,
\end{equation}
which are finite by Proposition~\ref{al_n-summable}.

\begin{prop}\label{J-j(1)}
We have
\[
\text{\em (a)} \, J^{(0)}_{A(r,R)}(1)=s^{r,R}_0,  \quad \text{\em (b)}\,J^{(1)}_{A(r,R)}(1) = \frac{s^{r,R}_0s^{r,R}_2-(s^{r,R}_1)^2}{s^{r,R}_0}, 
\]
and
\[\text{\em (c)}\,
J_{A(r,R)}^{(2)}(1)=\frac{(s^{r,R}_1)^2 s^{r,R}_4 - s^{r,R}_0s^{r,R}_2 s^{r,R}_4 -2 s^{r,R}_1s^{r,R}_2s^{r,R}_3 +s^{r,R}_0(s^{r,R}_3)^2+(s^{r,R}_2)^3} {(s^{r,R}_1)^2-s^{r,R}_0 s^{r,R}_2}.
\]
\end{prop}
\begin{proof}
Let $f\in H^2\big(bA(r,R)\big)$ and $f(z)=\sum_{n=-\infty}^{\infty} a_n z^n$. Then
\[
\|f\|^2_{H^2\big(bA(r,R)\big)}=\sum_{n=-\infty}^{\infty}\vert a_n\vert^2 \al_n^{r,R}.
\]
(a) By the Cauchy-Schwarz inequality
\begin{equation}\label{f(1)}
\big\vert f(1)\big\vert^2 =\left\vert \sum_{n=-\infty}^{\infty} a_n\right\vert^2 = \left\vert \sum_{n=-\infty}^{\infty} \frac{1}{\sqrt{\al_n^{r,R}}} \Big(a_n \sqrt{\al_n^{r,R}}\Big) \right\vert^2
 \leq  s^{r,R}_0 \|f\|^2_{H^2(bA(r,R))}.
\end{equation}
Since $f$ is arbitrary, this implies that
\[
J_{A(r,R)}^{(0)}(1) \leq s^{r,R}_0.
\]
Thus to complete the proof of (a), all we need to do is to produce a function in $H^2(bA(r,R))$ for which the inequality in \eqref{f(1)} is equality. In view of Proposition~\ref{al_n-summable},
\[
f_0(z)=\sum_{n=-\infty}^{\infty} \frac{z^n}{\al_n^{r,R}},
\]
is in $H^2(bA(r,R))$. Note that $\|f_0\|^2_{H^2(bA(r,R))}=s^{r,R}_0=f_0(1)$, and therefore,
\[
\big\vert f_0(1)\big\vert^2 =s^{r,R}_0 \|f_0\|^2_{H^2(bA(r,R))}.
\]
Thus, $f_0$ has the desired property and hence the proof of (a) follows.

(b) Assume that $f(1)=0$, so that $\sum_{n=-\infty}^{\infty} a_n=0$. Then for any $\be \in \mf{R}$,
\begin{multline}\label{f'(1)}
\big\vert f'(1)\big\vert^2 = \left \vert \sum_{n=-\infty}^{\infty} na_n\right\vert^2 = \left \vert \sum_{n=-\infty}^{\infty} (n-\be)a_n\right\vert^2\\ =  \left \vert \sum_{n=-\infty}^{\infty} \frac{n-\be}{\sqrt{\al_n^{r,R}}} \Big(a_n \sqrt{\al_n^{r,R}}\Big)\right\vert^2
 \leq \sum_{n=-\infty}^{\infty}\frac{(n-\be)^2}{\al_n^{r,R}} \|f\|_{H^2(bA(r,R))}^2,
\end{multline}
which shows that
\begin{equation}\label{J^1-ineq}
J^{(1)}_{A(r,R)}(1) \leq \sum_{n=-\infty}^{\infty}\frac{(n-\be)^2}{\al_n^{r,R}}.
\end{equation}
We now show that for a suitable $\be$, the above inequality is equality. It suffices to produce a function $f \in H^2(bA(r,R))$ such that $f(1)=0$, and for which the inequality in \eqref{f'(1)} is equality. Let $\be \in \mf{R}$ and set
\[
f_\beta(z)=\sum_{n=-\infty}^{\infty} \frac{n-\be}{\al_n^{r,R}}z^n.
\]
By Proposition~\ref{al_n-summable}, $f_\beta \in H^2(bA(r,R))$.  We choose $\be$ so that
\begin{equation}\label{f(1)=0}
f_\beta(1)=\sum_{n=-\infty}^{\infty} \frac{n-\be}{\al_n^{r,R}}=0.
\end{equation}
Indeed, we can take
\begin{equation}\label{beta}
\be=\frac{s^{r,R}_1}{s^{r,R}_0}.
\end{equation}
Then, by \eqref{f(1)=0},
\[
\|f_\beta\|_{H^2(bA(r,R))}^2= \sum_{n=-\infty}^{\infty} \frac{(n-\be)^2}{\al_n^{r,R}}=\sum_{n=-\infty}^{\infty} \frac{n(n-\be)}{\al_n^{r,R}}-\sum_{n=-\infty}^{\infty} \frac{\be(n-\be)}{\al_n^{r,R}}=\sum_{n=-\infty}^{\infty} \frac{n(n-\be)}{\al_n^{r,R}}=f_\beta'(1),
\]
from which it follows that
\[
\big\vert f_\beta'(1)\big\vert^2= \big\vert f_\beta'(1)\big\vert \|f_{\beta}\|^2_{H^2(bA(r,R))} = \sum_{n=-\infty}^{\infty}\frac{(n-\be)^2}{\al_n^{r,R}} \|f_\beta\|^2_{H^2(bA(r,R))},
\]
and hence $f_\beta$ has the desired property. It follows that
\[
J^{(1)}_{A(r,R)}(1)=\sum_{n=-\infty}^{\infty}\frac{(n-\be)^2}{\al_n^{r,R}}.
\]
Again, in view of \eqref{f(1)=0} and \eqref{beta}, we can write
\[
J_{A(r,R)}^{(1)}(1)=\sum_{n=-\infty}^{\infty}\frac{n(n-\be)}{\al_n^{r,R}} =s^{r,R}_2-\frac{s^{r,R}_1}{s^{r,R}_0} s^{r,R}_1 =\frac{s^{r,R}_0s^{r,R}_2-\big(s^{r,R}_1\big)^2}{s^{r,R}_0},
\]
which proves (b).

(c) Assume that $f(1)=f'(1)=0$ so that
\[
\sum_{n=-\infty}^{\infty} a_n = \sum_{n=-\infty}^{\infty} n a_n = 0.
\]
Then for any $\ga,\de \in \mf{R}$,
\begin{multline}\label{f''(1)}
\big \vert f''(1)\big\vert^2 = \left \vert\sum_{n=-\infty}^{\infty} n(n-1)a_n \right\vert^2 = \left \vert\sum_{n=-\infty}^{\infty} n^2a_n \right\vert^2= \left \vert\sum_{n=-\infty}^{\infty} (n^2-\ga n -\de) a_n \right\vert^2\\
= \left \vert\sum_{n=-\infty}^{\infty} \frac{n^2-\ga n -\de}{\sqrt{\al_n^{r,R}}} \Big( a_n \sqrt{\al_n^{r,R}}\Big)  \right\vert^2 \leq \left(\sum_{n=-\infty}^{\infty} \frac{(n^2-\ga n -\de)^2}{\al_n^{r,R}}\right) \|f\|^2_{H^2(bA(r,R))},
\end{multline}
which implies that
\begin{equation}\label{J^2-ineq}
J_{A(r,R)}^{(2)}(1) \leq \sum_{n=-\infty}^{\infty} \frac{(n^2-\ga n -\de)^2}{\al_n^{r,R}} .
\end{equation}
As before, we show that for a suitable choice of $\ga$ and $\de$, there is a function $f$ in $H^2(bA(r,R))$ satisfying $f(1)=f'(1)=0$, and for which the inequality in \eqref{f''(1)}  is equality. This would imply that the inequality \eqref{J^2-ineq} is also equality. For $\ga, \de \in \mf{R}$, set
\[
f_{\ga\de}(z)= \sum \frac{n^2-\ga n -\de}{\al_n^{r,R}}z^n.
\]
By Proposition~\ref{al_n-summable}, $f_{\ga\de}\in H^2(bA(r,R))$. We claim that there exist $\ga,\de$ such that

\begin{equation}\label{f=f'=0}
\begin{aligned}
f_{\ga\de}(1) &  =\sum_{n=-\infty}^{\infty} \frac{n^2-\ga n -\de}{\al_n^{r,R}}=0, \quad \text{and}\\
f'_{\ga\de}(1) & = \sum_{n=-\infty}^{\infty} \frac{n(n^2-\ga n -\de)}{\al_n^{r,R}}=0.
\end{aligned}
\end{equation}
The above equations can be written as
\begin{align*}
s^{r,R}_1 \ga + s^{r,R}_0 \de-s^{r,R}_2=0,\\
s^{r,R}_2 \ga + s^{r,R}_1 \de-s^{r,R}_3=0.
\end{align*}
Since,
\[
(s^{r,R}_1)^2-s^{r,R}_0s^{r,R}_2 =\left(\sum_{n=-\infty}^{\infty} \frac{n}{\al_n^{r,R}}\right)^2- \left(\sum_{n=-\infty}^{\infty} \frac{n^2}{\al_n^{r,R}}\right)\left(\sum_{n=-\infty}^{\infty} \frac{1}{\al_n^{r,R}}\right)<0,
\]
which is a consequence of the Cauchy-Schwarz inequality, the above system has a unique solution given by
\begin{equation}\label{ga-de}
\ga=\frac{s^{r,R}_1s^{r,R}_2-s^{r,R}_0s^{r,R}_3}{(s^{r,R}_1)^2-s^{r,R}_0s^{r,R}_2}, \quad \text{and} \quad \de=\frac{s^{r,R}_1s^{r,R}_3-(s^{r,R}_2)^2}{(s^{r,R}_1)^2-s^{r,R}_0 s^{r,R}_2  }.
\end{equation}
We now choose $\ga$ and $\de$ as above. Then, by \eqref{f=f'=0},
\[
\|f_{\ga\de}\|^2_{H^2(bA(r,R))}=\sum_{n=-\infty}^{\infty} \frac{(n^2-\ga n -\de)^2}{\al_n^{r,R}} =\sum_{n=-\infty}^{\infty} \frac{n^2(n^2-\ga n -\de)}{\al_n^{r,R}},
\]
and also
\[
f_{\ga\de}''(1)=\sum_{n=-\infty}^{\infty} \frac{n(n-1)(n^2-\ga n -\de)}{\al_n^{r,R}} =\sum_{n=-\infty}^{\infty} \frac{n^2(n^2-\ga n -\de)}{\al_n^{r,R}}.
\]
Therefore,
\[
\big \vert f_{\ga\de}''(1) \big \vert^2 = \big\vert f_{\ga\de}''(1) \big\vert \|f_{\ga\de}\|^2_{H^2(bA(r,R))},
\]
and thus $f_{\ga\de}$ has the desired property. It follows that
\[
J^{(2)}_{A(r,R)}(1)=\sum_{n=-\infty}^{\infty} \frac{(n^2-\ga n -\de)^2}{\al_n^{r,R}}.
\]
Again, in view of \eqref{f=f'=0} and \eqref{ga-de}, we can write
\begin{align*}
J^{(2)}_{A(r,R)}(1) &=\sum_{n=-\infty}^{\infty} \frac{n^2(n^2-\ga n -\de)}{\al_n^{r,R}}\\
& = s^{r,R}_4-\frac{s^{r,R}_1s^{r,R}_2-s^{r,R}_0s^{r,R}_3}{(s^{r,R}_1)^2-s^{r,R}_0s^{r,R}_2}\phi_3^{r,R}-\frac{s^{r,R}_1s^{r,R}_3-(s^{r,R}_2)^2}{(s^{r,R}_1)^2-s^{r,R}_0 s^{r,R}_2 } \phi_2^{r,R}\\
& = \frac{(s^{r,R}_1)^2 s^{r,R}_4 - s^{r,R}_0s^{r,R}_2 s^{r,R}_4 -2 s^{r,R}_1s^{r,R}_2s^{r,R}_3 +s^{r,R}_0(s^{r,R}_3)^2+(s^{r,R}_2)^3} {(s^{r,R}_1)^2-s^{r,R}_0 s^{r,R}_2},
\end{align*}
as required. This completes the proof of (c) and the proposition.
\end{proof}
We now focus on the annulus $A_r=\{z \in \mf{C}:0< \vert z \vert <1\}$. Note that by the transformation formula \eqref{tr Js}, we have
\begin{equation}\label{J_A_r-J_A}
r^{(2j+1)\lambda}J_{A_{r}}^{(j)}(r^{\lambda})=J_{A(r^{1-\lambda},r^{-\lambda})}^{(j)}(1), \quad j\geq 0.
\end{equation}
Observe that while on the left-hand side, both the domains and points vary, on the right-hand side only the domains vary. Hence, studying the asymptotic behaviour of the right-hand side is relatively easier and in view of Proposition~\ref{J-j(1)}, it now suffices to analyse the quantities $s_j^{r^{1-\la}, r^{-\la}}$. For simplicity, let us write $\al_n=\al_{n}^{r^{1-\la}, r^{-\la}}$, and $s_j=s_j^{r^{1-\la}, r^{-\la}}$. Note from \eqref{H-norm-z^n} that for $n \geq 0$,
\[
\al_n=2\pi \frac{1+r^{2n+1}}{r^{(2n+1)\la}} \quad \text{and} \quad \al_{-n-1}=2\pi \frac{1+r^{2n+1}}{r^{(2n+1)(1-\la)}}.
\]
Thus, we can write
\begin{equation}\label{phi-j-A_r}
\begin{aligned}
s_j & = \sum_{n=-\infty}^{\infty} \frac{n^j}{\al_n}\\
& = \sum_{n=0}^{\infty}  \left(\frac{n^j}{\al_n} + \frac{(-1)^j(n+1)^j}{\al_{-n-1}}\right)\\
& = \frac{1}{2\pi} \sum_{n=0}^{\infty}\frac{1}{1+r^{2n+1}} \big( n^j r^{(2n+1)\la}+ (-1)^j(n+1)^j r^{(2n+1)(1-\la)}\big)\\
& = u_j +v_j,
\end{aligned}
\end{equation}
where $u_j$ is the $0$th term of the series and $v_j$ is the sum of the terms from $n=1$ onwards. Observe that
\begin{equation}\label{u_j}
u_0=\frac{1}{2\pi} \frac{r^{\la}+r^{1-\la}}{1+r}, \quad u_j=(-1)^j\frac{1}{2\pi} \frac{r^{1-\la}}{1+r}, \quad \text{for } j \geq 1,
\end{equation}
and
\begin{equation}\label{v_j-s_j}
v_j  = \frac{1}{2\pi}\frac{r^{3\la}+(-1)^j2^jr^{3(1-\la)}}{1+r^3}+O(r^{5\la})+O(r^{5(1-\la)}).
\end{equation}
The following lemma that describes the asymptotic behaviour of the maximal domain functions associated with $A_r$ is the key to the proof of Theorem~\ref{variation}. In what follows, for $\phi,\psi:(0,1) \to (0,\infty)$, we will write $\phi(r)\sim\psi(r)$ if for all sufficiently small $\epsilon>0$,
\[
\phi(r)-\psi(r)=\psi(r)o(r^{\epsilon}).
\]

\begin{lem}\label{asym curv}
Let $r,~\lambda\in (0,1)$.  Then
\[
\text{\em (a)}\, r^{\lambda}J_{A_{r}}^{(0)}(r^{\lambda}) = N_{\la}^{(0)}(r),\quad
\text{\em (b)} \, r^{3\lambda}J_{A_{r}}^{(1)}(r^{\lambda})=\frac{N_{\la}^{(1)}(r)}{N_{\la}^{(0)}(r)},\quad
\text{\em (c)} \, r^{5\lambda}J_{A_{r}}^{(2)}(r^{\lambda})=\frac{N_{\la}^{(2)}(r)}{N_{\la}^{(1)}(r)},
\] 
where $N_{\la}^{(j)}(r)$ are functions with
\begin{align*}
N_{\la}^{(0)}(r) & \sim \frac{1}{2\pi} \frac{r^{\lambda}+r^{1-\lambda}}{1+r},\\
N_{\la}^{(1)}(r) & \sim \frac{1}{4\pi^2} \frac{r}{(1+r)^2}+\frac{1}{4\pi^2} \frac{(r^{4\la}+r^{4(1-\la)})+4r(r^{2\la}+r^{2(1-\la)})}{(1+r)(1+r^3)}, \quad \text{and}\\
N_{\la}^{(2)}(r)&\sim \frac{1}{2\pi^3}\frac{r^{9\lambda}+r^{9(1-\lambda)}}{(1+r)(1+r^{3})(1+r^{5})}+\frac{1}{2\pi^3}\frac{r\big(r^{3\lambda}+r^{3(1-\lambda)}\big)}{(1+r)^{2}(1+r^{3})}.
      \end{align*}
\end{lem}

\begin{proof} 
(a) By \eqref{J_A_r-J_A} and Proposition~\ref{J-j(1)}, we have
\[
r^{\la} J^{(0)}_{A_r}(r^{\la})=s_0. 
\]
We define $N_{\la}^{(0)}(r)=s_0$. Then from \eqref{phi-j-A_r}--\eqref{v_j-s_j},
\[
N_{\la}^{(0)}(r)=u_0+v_0=\frac{1}{2\pi} \frac{r^{\la}+r^{1-\la}}{1+r}+E(r),
\]
where
\[
E(r)=O\big(r^{3\la}\big)+O\big(r^{3(1-\la)}\big).
\]
Thus, for all sufficiently small $\ep>0$,
\[
E(r)=o\big(r^{2\la+\ep}\big)+o\big(r^{2(1-\la)+\ep}\big).
\]
Therefore, to complete the proof of (a), we only need to show that
\[
\frac{E(r)}{\frac{1}{2\pi} \frac{r^{\la}+r^{1-\la}}{1+r}}=o(r^\ep).
\]
But
\[
\frac{\big\vert o(r^{2\la+\ep})\big\vert }{\frac{1}{2\pi} \frac{r^{\la}+r^{1-\la}}{1+r}r^{\ep}}\leq \frac{2\pi(1+r)\big\vert o(r^{2\la+\ep})\big\vert}{r^{\la+\ep}} \to 0
\]
as $r \to 0$, and by symmetry, the same holds for the $o(r^{2(1-\la)+\ep})$ term.

(b) By \eqref{J_A_r-J_A} and Proposition~\ref{J-j(1)}, we have
\[
r^{3\la}J^{(1)}_{A_r}(r) =J^{(1)}_{A(r^{1-\la},r^{-\la})}(1) = \frac{s_0s_2-s_1^2}{s_0} = \frac{N_{\la}^{(1)}(r)}{N_{\la}^{(0)}(r)},
\]
where $N_{\la}^{(1)}(r)=s_0s_2-s_1^2$. The asymptotic behaviour of $N^{(0)}_{\la}(r)$ is already discussed in (a). For $N_{\la}^{(1)}(r)$, note that by \eqref{phi-j-A_r},
\begin{equation}\label{s0s2-s1^2}
N_{\la}^{(1)}(r) = (u_0u_2-u_1^2)+(u_0v_2+u_2v_0-2u_1v_1)+(v_0v_2-v_1^2)=I+II+III,
\end{equation}
say. From \eqref{u_j} and \eqref{v_j-s_j}, we have
\begin{equation}\label{u-term}
I=\frac{1}{4\pi^2} \frac{r}{(1+r)^2},
\end{equation}
and
\begin{multline}\label{uv-term}
II=  \frac{1}{2\pi} \frac{r^{\la}+r^{1-\la}}{1+r} \cdot \frac{1}{2\pi} \frac{r^{3\la}+4r^{3(1-\la)}}{1+r^3}+\frac{1}{2\pi}\frac{r^{1-\la}}{1+r} \cdot \frac{1}{2\pi} \frac{r^{3\la}+r^{3(1-\la)}}{1+r^3} \\
+2\frac{1}{2\pi}\frac{r^{1-\la}}{1+r} \cdot \frac{1}{2\pi} \frac{r^{3\la}-2r^{3(1-\la)}}{1+r^3}+F_1(r)\\
=  \frac{1}{4\pi^2}\frac{(r^{4\la}+r^{4(1-\la)}) + 4r(r^{2\la}+r^{2(1-\la)})}{(1+r)(1+r^3)}+F_1(r),
\end{multline}
where
\[
F_1(r)=O(r^{4\la+1})+O(r^{4(1-\la)+1})+O(r^{6\la})+O(r^{6(1-\la)}).
\]
Finally,
\begin{equation}\label{v-term}
\begin{aligned}
III & =\frac{1}{2\pi} \frac{r^{3\la}+r^{3(1-\la)}}{1+r^3} \cdot \frac{1}{2\pi} \frac{r^{3\la}+4r^{3(1-\la)}}{1+r^3}-\left(\frac{1}{2\pi} \frac{r^{3\la}-2r^{3(1-\la)}}{1+r^3}\right)^2+F_2(r)\\
& = \frac{1}{4\pi^2} \frac{9 r^3 }{(1+r^3)^2} +F_2(r),
\end{aligned}
\end{equation}
where
\[
F_2(r)=O\big(r^{8\la}\big)+O\big(r^{8(1-\la)}\big)+O\big(r^{3+2(1-\la)}\big) + O\big(r^{3+2\la}\big)+O\big(r^{5}\big).
\]
It follows from \eqref{s0s2-s1^2}--\eqref{v-term} that
\[
N_{\la}^{(1)}(r)=A(r)+F(r),
\]
where
\[
A(r) = \frac{1}{4\pi^2} \frac{r}{(1+r)^2}+\frac{1}{4\pi^2} \frac{(r^{4\la}+r^{4(1-\la)})+4r(r^{2\la}+r^{2(1-\la)})}{(1+r)(1+r^3)},
\]
and
\[
F(r)=o(r^2)+o(r^{4\la + \ep})+o(r^{4(1-\la)+\ep}),
\] 
for all sufficiently small $\ep$.
Thus, to complete the proof of (b), we only need to show that
\[
F(r)/A(r)=o(r^{\ep}).
\]
Note that
\[
\frac{\big\vert o(r^2)\big\vert}{r^{\ep}A(r)}\leq \frac{\big\vert o(r^2)\big\vert}{r^{\ep}\frac{1}{4\pi^2} \frac{r}{(1+r)^2}} \to 0
\]
as $r \to 0$. Similarly,
\[
\frac{\big\vert o(r^{4\la+\ep})\big\vert}{r^{\ep}A(r)} \leq \frac{\big\vert o(r^{4\la+\ep})\big\vert}{r^{\ep}\frac{1}{4\pi^2}\frac{r^{4\la}}{(1+r)(1+r^3)}} \to 0
\]
as $r \to 0$, and by symmetry, the same holds for the $o(r^{4(1-\la)+\ep})$ term.

(c) By \eqref{J_A_r-J_A} and Proposition~\ref{J-j(1)}, we have
\[
r^{5\la} J^{(2)}_{A_r}(r^{\la})=J^{(2)}_{A(r^{1-\la},r^{-\la})}(1) = \frac{-s_1^2s_4 + s_0s_2 s_4 +2 s_1s_2s_3 -s_0s_3^2-s_2^3}{s_0s_2-s_1^2}=\frac{N_{\la}^{(2)}(r)}{N_{\la}^{(1)}(r)},
\]
where 
\[
N_{\la}^{(2)}(r)=-s_1^2s_4 + s_0s_2 s_4 +2 s_1s_2s_3 -s_0s_3^2-s_2^3.
\]
We have already obtained the asymptotic behaviour of $N_{\la}^{(1)}(r)$ in (b). For $N_{\la}^{(2)}(r)$, a similar but only lengthier calculation as above shows that
\[
N_{\la}^{(2)}(r)=B(r)+G(r),
\]
where
\[
B(r)=\frac{1}{2\pi^3}\frac{r^{9\la}+r^{9(1-\la)}}{(1+r)(1+r^{3})(1+r^{5})}+\frac{1}{2\pi^3}\frac{r(r^{3\la}+r^{3(1-\la)})}{(1+r)^{2}(1+r^{3})},
\]
and
\[
G(r)=o\big(r^{10\la+\ep}\big)+o\big(r^{10(1-\la)+\ep}\big)+o\big(r^{1+3\la+\ep}\big)+o\big(r^{1+3(1-\la)+\ep}\big).
\]
Thus, to complete the proof of (c), we only need to check that
\[
\frac{G(r)}{B(r)}=o(r^{\ep}),
\]
which can be done in a similar way as above. This completes the proof of (c) and the lemma.
\end{proof}
We are now ready for the proof of Theorem~\ref{variation}.

\begin{proof}[Proof of Theorem~\ref{variation}]
(a) By \eqref{berg-fuks-formula} and Lemma~\ref{asym curv} (a), we have
\[
c_{A_r}(r^{\la})=2\pi S_{A_r}(r^{\la})=2\pi J^{(0)}_{A_r}(r^{\la})=2\pi \frac{1}{r^{\la}} N_{\la}^{(0)}(r)= \frac{1+r^{1-2\la}}{1+r}\big(1+o(r^{\ep})\big),
\]
from which (a) follows immediately.

\medskip

(b) By \eqref{berg-fuks-formula} and Lemma~\ref{asym curv}, we have
\[
s^2_{A_r}(r^{\la})=\frac{J^{(1)}(r^\la)}{J^{(0)}(r^{\la})} = \frac{1}{r^{2\la}} \frac{N_{\la}^{(1)}(r)}{\big(N_{\la}^{(0)}(r)\big)^2}=
\frac{r+\frac{1+r}{1+r^3}\Big(r^{4\la}+r^{4(1-\la)}+4r\big(r^{2\la}+r^{2(1-\la)}\big)\Big)}{r^{2\la}\big(r^{\la}+r^{1-\la}\big)^2} f(r),
\]
where $f(r)$ is a positive function such that $f(r) \to 1$ as $r \to 0+$. Dividing the numerator and denominator by $r^{4\la}$, we write
\[
s^2_{A_r}(r^{\la})=\frac{r^{1-4\la}+\frac{1+r}{1+r^3}\Big(1+r^{4(1-2\la)}+4\big(r^{1-2\la}+r^{3(1-2\la)}\big)\Big)}{\big(1+r^{1-2\la}\big)^2}f(r),
\]
from which the limiting behaviour on $(0,1/4]$ follows immediately. On the other hand, by dividing the numerator and denominator by $r^2$, we write
\[
s^2_{A_r}(r^{\la})=\frac{\frac{1}{r}+\frac{1+r}{1+r^3}\Big(r^{2(2\la-1)}+r^{2(1-2\la)}+4\big(r^{2\la-1}+r^{1-2\la}\big)\Big)}{\big(r^{2\la-1}+1\big)^2}f(r),
\]
from which the limiting behaviour on $(1/4,1)$ follows. Thus the proof of (b) is complete.

\medskip

(c) First, note that
\[
\kappa_{c_{A_r}}=-\frac{\De \log c_{A_r}}{c^2_{A_r}}=-4\frac{\pa\ov\pa \log c_{A_r}}{c^2_{A_r}}=-\frac{\pa\ov\pa \log S_{A_r}}{\pi^2 S^2_{A_r}}=-\frac{1}{\pi^2}\frac{s^2_{A_r}}{S^2_{A_r}}=-\frac{1}{\pi^2}\frac{J_{A_r}^{(1)}}{\big(J_{A_r}^{(0)}\big)^3},
\]
using \eqref{berg-fuks-formula}. Now, using Lemma~\ref{asym curv}, we have
\[
\kappa_{c_{A_r}}(r^{\la})=-\frac{1}{\pi^2}\frac{N_{\la}^{(1)}(r)}{\big(N_{\la}^{(0)}(r)\big)^4}=-4\frac{r(1+r)^2+\frac{(1+r)^3}{1+r^3}\Big(r^{4\la}+r^{4(1-\la)}+4r\big(r^{2\la}+r^{2(1-\la)}\big)\Big)}{\big(r^{\la}+r^{1-\la}\big)^4} g(r),
\]
where $g(r)$ is a positive function such that $g(r) \to 1$ as $r \to 0+$. Observe that the right-hand side is symmetric in $\la$ and $1-\la$, and hence enough to compute the limiting behaviour on $(0,1/2]$. Dividing the numerator and denominator by $r^{4\la}$, we write
\[
\kappa_{c_{A_r}}(r^{\la})=-4\frac{r^{1-4\la}(1+r)^2+\frac{(1+r)^3}{1+r^3}\Big(1+r^{4(1-2\la)}+4\big(r^{1-2\la}+r^{3(1-2\la)}\big)\Big)}{(1+r^{1-2\la})^2}g(r),
\]
from which the limiting behaviour on $(0,1/2]$ (and hence on $(0,1)$) follows immediately. This completes the proof of (c).

\medskip

(d) Recall from \eqref{berg-fuks-formula} that
\begin{equation}\label{curv-s-A_r}
\kappa_{s_{A_{r}}}(r^{\la})=4-2\frac{J_{A_{r}}^{(0)}(r^{\la})J_{A_{r}}^{(2)}(r^{\la})}{\big(J_{A_{r}}^{(1)}(r^{\la})\big)^{2}}.
\end{equation}
Therefore, using Lemma~\ref{asym curv}, we have
\begin{multline*}
\kappa_{s_{A_{r}}}(r^{\la})=4-2\frac{\big(N_{\la}^{(0)}(r)\big)^3 N_{\la}^{(2)}(r)}{\big(N_{\la}^{(1)}(r)\big)^3}\\
=4-8 \frac{\left(\frac{r^{\la}+r^{1-\la}}{1+r}\right)^3 \left(\frac{r^{9\lambda}+r^{9(1-\lambda)}}{(1+r)(1+r^{3})(1+r^{5})}+\frac{r\big(r^{3\lambda}+r^{3(1-\lambda)}\big)}{(1+r)^{2}(1+r^{3})}\right)}{\left( \frac{r}{(1+r)^2}+ \frac{(r^{4\la}+r^{4(1-\la)})+4r(r^{2\la}+r^{2(1-\la)})}{(1+r)(1+r^3)}\right)^3}h(r),
\end{multline*}
where $h(r)$ is a positive function such that $h(r) \to 1$ as $r \to 0+$. Observe that the right-hand side is symmetric with respect to $\la$ and $1-\la$ and hence it is enough to compute the limiting behaviour on $(0,1/2]$. Moreover, terms like $1+r$ etc. tend to $1$ as $r\to 0+$ and so they can be ignored while computing the limit. In other words, the limit of $\kappa_{s_{A_r}}(r^{\la})$ as $r \to 0+$ is the same as that of
\[
4-8\frac{\big(r^{\la}+r^{1-\la}\big)^3 \Big(r^{9\lambda}+r^{9(1-\lambda) }+r\big(r^{3\la}+r^{3(1-\la)}\big)\Big)}{\Big(r+ \big(r^{4\la}+r^{4(1-\la)}\big)+4r\big(r^{2\la}+r^{2(1-\la)}\big)\Big)^3},
\]
as $r \to 0+$. Now, dividing the numerator and the denominator of the second term by $r^{12 \la}$, the above expression is equal to
\[
4-8\frac{\big(1+r^{1-2\la}\big)^3\big(1+r^{9(1-2\lambda) }+r^{1-6\la}+r^{4(1-3\la)}\big)}{\Big(r^{1-4\la}+ 1+r^{4(1-2\la)}+4\big(r^{1-2\la}+r^{3(1-2\la)}\big)\Big)^3},
\]
from which the limiting behaviour on $(0, 1/4]$ follows. Multiplying the numerator and the denominator of the second term of the above expression by $r^{3(4\la-1)}$, we obtain
\[
4-8\frac{(1+r^{1-2\la})^3\big(r^{3(4\la-1)}+r^{6(1-\lambda) }+ r^{2(3\la-1)}+r\big)}{\Big(1+ r^{4\la-1}+r^{3-4\la}+4\big(r^{2\la}+r^{2(1-\la)}\big)\Big)^3},
\]
from which the limiting behaviour on $(1/4, 1/2]$ follows. As mentioned above, this gives us the limiting behaviour on all of $(0,1)$. This completes the proof of (d) and the theorem.
\end{proof}

\medskip

\noindent \textbf{Acknowledgments:} The authors would like to thank Prachi Mahajan and Kaushal Verma for several fruitful discussions throughout the course of this work. A. Bhatnagar is supported in part by a CSIR-UGC fellowship (Fellowship No. 1003). D. Borah is supported in part by an ANRF grant (Grant No. MTR/2021/000377).

\end{document}